\documentclass[reqno]{amsart}
\usepackage{amscd, amsfonts, amsmath,amssymb,amsthm, mathtools,mathrsfs} 
\usepackage{tabularx,ltablex}
\usepackage[all,color]{xy}
\usepackage[colorlinks=true,
linkcolor=black,
citecolor=black,
filecolor=black,
urlcolor=black,
bookmarks=true,
bookmarksopen=true,
bookmarksopenlevel=3,
plainpages=false,
pdfpagelabels=true]{hyperref}
\usepackage{cleveref}
\usepackage{latexsym, bbold}
\usepackage{aligned-overset}
\usepackage{microtype}
\usepackage[margin=1in]{geometry}

\allowdisplaybreaks

\newtheorem{introques}{Question}
\newtheorem{introdefn}[introques]{Definition}
\newtheorem{introdethm}[introques]{Theorem}

\newtheorem{thm}{Theorem}[section]

\newtheorem{defn}[thm]{Definition}
\newtheorem{Ex}[thm]{Example} 
\newtheorem{lemma}[thm]{Lemma} 
\newtheorem{proposition}[thm]{Proposition} 
\newtheorem{remark}[thm]{Remark} 
\newtheorem{Cor}[thm]{Corollary}

\numberwithin{thm}{subsection}

\DeclareMathOperator{\GL}{GL}
\DeclareMathOperator{\id}{id}
 
\DeclareMathOperator{\Ext}{Ext}
\DeclareMathOperator{\Aut}{Aut}

\DeclareMathOperator{\gr}{gr}

\newcommand{\kk}{\Bbbk}
\newcommand{\uqba}{\underline{\operatorname{end}}}
\newcommand{\mc}{\mathcal}

\newcommand{\uaut}{\underline{\rm aut}}

\numberwithin{equation}{subsection}

\makeatletter
\renewcommand\subsection{\@startsection{subsection}{2}
  \z@{-.5\linespacing\@plus-.7\linespacing}{.5\linespacing}
  {\bfseries}}

\title[Twisting of graded quantum groups and solutions to the quantum Yang-Baxter equation]{Twisting of graded quantum groups \\ and solutions to the quantum Yang-Baxter equation}  
\author[Huang]{Hongdi Huang}
\address{(Huang) Department of Mathematics, Rice University, Houston, TX 77005, U.S.A.}
\email{h237huan@rice.edu}

\author[Nguyen]{Van C. Nguyen}
\address{(Nguyen) Department of Mathematics, United States Naval Academy, Annapolis, MD 21402, U.S.A.}
\email{vnguyen@usna.edu}

\author[Ure]{Charlotte Ure}
\address{(Ure) Department of Mathematics, University of Virginia, Charlottesville, VA 22904, U.S.A.}
\email{cu9da@virginia.edu}

\author[Vashaw]{Kent B. Vashaw}
\address{(Vashaw) Department of Mathematics,
Massachusetts Institute of Technology,
Cambridge, MA 02139, U.S.A.}
\email{kentv@mit.edu}

\author[Veerapen]{Padmini Veerapen}
\address{(Veerapen) Department of Mathematics, Tennessee Tech University, Cookeville, TN 38505, U.S.A.}
\email{pveerapen@tntech.edu}

\author[Wang]{Xingting Wang}
\address{(Wang) Department of Mathematics, Howard University, 2400 6th St NW, Washington, DC 20059, U.S.A.}
\email{xingting.wang@howard.edu}

\date\today


\begin{document}
\begin{abstract} 
Let $H$ be a Hopf algebra that is $\mathbb Z$-graded as an algebra. We provide sufficient conditions for a 2-cocycle twist of $H$ to be a Zhang twist of $H$. In particular, we introduce the notion of a twisting pair for $H$ such that the Zhang twist of $H$ by such a pair is a 2-cocycle twist. We use twisting pairs to describe twists of Manin's universal quantum groups associated to quadratic algebras and provide twisting of solutions to the quantum Yang-Baxter equation via the Faddeev-Reshetikhin-Takhtajan construction.
\end{abstract} 

\maketitle  

%%%%%%%%%%%%%%%%%%%%%%%%%%%%%%%%%%%%%%%%%
\section*{Introduction} 

In the theory of quantum groups, motivated by the tensor equivalence of two module categories, Drinfeld \cite{Dr87} introduced the notion of a Drinfeld twist which deforms the coalgebra structure of a Hopf algebra. In \cite{AEGN02}, Aljadeff, Etingof, Gelaki, and Nikshych discussed twists and properties of Hopf algebras invariant under Drinfeld twists. The dual version of the Drinfeld twist, called a 2-cocycle twist, was studied by Doi and Takeuchi \cite{Doi93,DT94}. This can be viewed as a deformation of the algebraic structure of a Hopf algebra and it yields a tensor equivalence of two corresponding comodule categories. 

Meanwhile, the notion of a twist of a graded algebra $A$ was introduced by Artin, Tate, and Van den Bergh in \cite{ATV1991} as a deformation of the original graded product of $A$; this definition was generalized by Zhang in \cite{Zhang1996}. One of the main applications of this twisting is that the graded module category of the twisted algebra is equivalent to that of the original algebra. Later, this twisting of a graded algebra became known as a Zhang twist, and ever since, it has played a pivotal role in various areas of noncommutative algebra and noncommutative projective geometry. For instance, many fundamental properties are preserved under Zhang twist, such as the noncommutative projective schemes, Gelfand-Kirillov dimension and global dimension, graded $\Ext$, Artin--Schelter regularity, and point modules \cite{Ro2016,Zhang1996}.

There are other notions of (cocycle) twists of noncommutative algebras (see e.g., \cite{Davies2017, M2005}). In this paper, we focus on the Zhang twists and 2-cocycle twists of $\mathbb Z$-graded algebras described above (see \Cref{sec:twists} for details). The comparison of these two twists and their corresponding algebraic structures is listed below. 
\begin{table}[h]
    \centering
    \begin{tabular}{c||c|c}
         & 2-cocycle twist &   Zhang twist \\
         \hline \hline
    algebra structure      & Hopf algebra &  $\mathbb Z$-graded algebra \\
    \hline
    twisting by  & 2-cocycle &  graded automorphism\\
    \hline 
    equivalence & tensor categories of comodules &   graded module categories\\
    \hline 
    examples & quantized coordinate rings &  Artin-Schelter regular algebras \\
    \end{tabular}
\end{table}

In practice, it is difficult to classify all 2-cocycles on a Hopf algebra and to understand the structures of their corresponding 2-cocycle twists. On the other hand, many ring-theoretic and homological properties are known to be preserved under Zhang twists due to the equivalence of their graded module categories. Here, we are interested in the following question relating Zhang twists to 2-cocycle twists of a graded Hopf algebra. 
\begin{introques}
\label{ques:twist}
When can a 2-cocycle twist of a Hopf algebra $H$ be given by a Zhang twist? 
\end{introques}

We partially answer these questions by providing sufficient conditions for a Zhang twist of a graded Hopf algebra to be a 2-cocycle twist, see \Cref{twist-and-2-cocycle}. To that end, we introduce the following twisting conditions on bialgebras.

\begin{introdefn}[Twisting Conditions]
\label{defn:twisting conditions}
A bialgebra $(B,m,u,\Delta,\varepsilon)$ satisfies the \emph{twisting conditions} if 
\begin{enumerate}
\item[(\textbf{T1})] as an algebra $B=\bigoplus_{n\in \mathbb{Z}} B_n$ is $\mathbb{Z}$-graded, and
\item[(\textbf{T2})] the comultiplication satisfies $\Delta(B_n)\subseteq B_n\otimes B_n$ for all $n\in \mathbb Z$. 
\end{enumerate}
\end{introdefn}

Equivalent descriptions of twisting conditions were discussed earlier in \cite[Lemma 1.3]{Bichon-Neshveyev-Yamashita2016}. Given these twisting conditions, the resulting twisted algebra can still be equipped with a bialgebra or Hopf algebra structure.
\begin{introdethm}[\Cref{zhang-twist-of-bialgebra}]
Let $B$ be a bialgebra (resp.~Hopf algebra) satisfying the twisting conditions (\textbf{T1})-(\textbf{T2}) in \Cref{defn:twisting conditions}. For any graded bialgebra (resp.~Hopf algebra) automorphism $\phi$ of $B$, the Zhang twist $B^{\phi}$ is again a bialgebra (resp.~Hopf algebra) satisfying the twisting conditions (\textbf{T1})-(\textbf{T2}). 
\end{introdethm}

It is important to point out that any bialgebra satisfying the twisting conditions (\textbf{T1})-(\textbf{T2}) is only graded as an algebra but not as a coalgebra. 
In later sections, we will discuss various examples of bialgebras that satisfy the twisting conditions, among which are quotients of free algebras generated by the generators of matrix coalgebras. There are two main sources of such examples: one is Manin's universal bialgebra that universally coacts on a quadratic algebra \cite{Manin2018}, and the other one is the Faddeev-Reshetikhin-Takhtajan (FRT) construction from a solution to the quantum Yang-Baxter equation \cite{CMZ, FRT}. In \Cref{t1-t3-hopf-env}, we further show that the Hopf envelope $\mathcal H(B)$ of any bialgebra $B$ that satisfies the twisting conditions (\textbf{T1})-(\textbf{T2}) possesses the same twisting conditions (\textbf{T1})-(\textbf{T2}) and its antipode satisfies the additional twisting condition (\textbf{T3}) in \Cref{cor:T3}, that is $S(\mathcal H(B)_n) \subseteq \mathcal H(B)_{-n}$ for all $n \in \mathbb{Z}$.
As a consequence, all the aforementioned cases provide us with examples of $\mathbb Z$-graded Hopf algebras through their Hopf envelopes, where we are able to deform the Hopf structures by considering a suitable Zhang twist. 

Our next goal is to compare the notions of Zhang twists and 2-cocycle twists for a $\mathbb Z$-graded Hopf algebra that satisfies the twisting conditions (\textbf{T1})-(\textbf{T2}). Though generally it is difficult to detect whether a 2-cocycle twist is indeed a Zhang twist, we are able to give sufficient conditions for a Zhang twist by a twisting pair to be a 2-cocycle twist. We provide the following definition for an arbitrary bialgebra.

\begin{introdefn}[Twisting Pair]
\label{defn:twisting pair}
Let $(B,m,u,\Delta,\varepsilon)$ be a bialgebra. A pair $(\phi_1,\phi_2)$ of algebra automorphisms of $B$ is said to be a \emph{twisting pair} if the following conditions hold:
\begin{enumerate}
\item[(\textbf{P1})] $\Delta\circ \phi_1=(\id \otimes \phi_1)\circ \Delta$ and $\Delta\circ \phi_2=(\phi_2\otimes \id)\circ \Delta$, and  
 \item[(\textbf{P2})] $\varepsilon \circ (\phi_1 \circ \phi_2)=\varepsilon$.
\end{enumerate} 
\end{introdefn} 

We can understand these twisting pairs as generalizations of conjugate actions from groups to quantum groups in view of \Cref{lem:twistingAG} or pairs of right and left winding endomorphisms, in the terminology of \cite{BZ08}, given by algebra homomorphisms from $B$ to the base field (see \Cref{thm:twistingpair}). In particular, \Cref{lem:winding} shows that for a twisting pair $(\phi_1,\phi_2)$ of $B$, $\phi_1$ and $\phi_2$ are uniquely determined by each other as winding automorphisms.  When a bialgebra or a Hopf algebra satisfies the twisting conditions (\textbf{T1})-(\textbf{T2}), \Cref{remark:twisting} (2) asserts that its algebra grading is always preserved by any twisting pair. Therefore, we may consider the Zhang twist of a bialgebra or a Hopf algebra satisfying the twisting conditions by an arbitrary twisting pair. Next, note that for a bialgebra $B$, condition (\textbf{P1}) is to ensure that $\phi_1$ and $\phi_2$ are right and left $B$-comodule maps from $B$ to itself, where $B$ is viewed as a $B$-bicomodule over itself via its comultiplication map $\Delta: B\to B\otimes B$. \Cref{twist-and-cleft} shows that Zhang twists of a Hopf algebra by graded automorphisms satisfying (\textbf{P1}) will give us cleft objects over the original Hopf algebra. Indeed, we show the following.

\begin{introdethm}
[\Cref{twist-and-2-cocycle} and \Cref{thm:Zhang and 2-cocycle twists}]
Let $H$ be a Hopf algebra satisfying the twisting conditions (\textbf{T1})-(\textbf{T2}) in \Cref{defn:twisting conditions}. For any twisting pair $(\phi_1,\phi_2)$ of $H$, the Zhang twist $H^{\phi_1 \circ \phi_2}$ is isomorphic to a 2-cocycle twist $H^\sigma$, where the 2-cocycle $\sigma: H \otimes H \to \kk$ is explicitly given by $\phi_1$ and $\phi_2$.
\end{introdethm}

This result provides an alternate approach to \cite[Theorem 2.8 and Remark 2.9]{Bichon-Neshveyev-Yamashita2016}, which is formulated in the context of twists of graded categories. While \cite{Bichon-Neshveyev-Yamashita2016} does not use the language of Zhang twists, the connection is highlighted in \cite[Remark 2.4]{Bichon-Neshveyev-Yamashita2018}. 

An advantage of using twisting pairs here is that these pairs are straightforward to compute when the Hopf algebras are presented by explicit generators and relations (see \Cref{thm:twistingpair}). Moreover, any twisting pair of a graded Hopf algebra $H$ gives rise to a 2-cocycle of $H$ through a simple formula. Since Zhang twists preserve graded module categories, using our results, we may deduce that many algebraic properties are maintained under 2-cocycle twists given by these twisting pairs. So for the remainder of the paper, we focus on computing possible twisting pairs for various Hopf algebras that satisfy the twisting conditions (\textbf{T1})-(\textbf{T2}).

Our first class of examples is the family of universal quantum groups from Manin's discussion of the possibility of ``hidden symmetry" in noncommutative projective algebraic geometry (\Cref{sec:Manin}). In \cite{Manin2018}, Manin showed that certain universal quantum groups can coact on the homogeneous coordinate ring of an (embedded) projective variety, and these quantum groups are typically much larger than the honest automorphism groups of the variety. Recently, these universal quantum groups have been extensively studied from the point of view of noncommutative invariant theory \cite{BDV13, Chirvasitu-Walton-Wang2019,Dubois-Violette2005,WW2016} and their corepresentation theories are investigated through the Tannaka-Krein formalism \cite{vdb2017}. Our result is as follows.

\begin{introdethm}[\Cref{thm:twist-aut} and \Cref{2cocycle-twist}]
Every twisting pair of the universal quantum group $\uaut(A)$ associated to a quadratic algebra $A$ is given by a graded algebra automorphism of $A$ in an explicit way. As a consequence, the corepresentation theory of $\uaut(A)$ only depends on the graded module cateogory over $A$. 
\end{introdethm}

Our next focus is the twisting of solutions to the quantum Yang-Baxter equation (QYBE) (see e.g., \cite{Dr87,T88}). The quantum Yang-Baxter equation originates in statistical mechanics. Solutions to this equation are building blocks for many invariants of knots, links and 3-manifolds. We provide a strategy to twist a solution $R$ to the QYBE by applying twisting pairs of the universal bialgebra $A(R)$ obtained via the FRT construction. We show that $A(R)$ is indeed a quadratic bialgebra that satisfies the twisting conditions (\textbf{T1})-(\textbf{T2}). Moreover, we are able to classify all twisting pairs of $A(R)$ and lift these twisting pairs to those of the Hopf envelope $\mathcal H(A(R))$. It turns out that all such twisting pairs can be determined explicitly in terms of the matrix coalgebra generators of $A(R)$ (see \Cref{Af-twisting-pair}). By writing out the 2-cocycles on $\mathcal H(A(R))$ explicitly using these twisting pairs, we then are able to twist solutions to the quantum Yang-Baxter equation (see \Cref{cor:TwistR}). In the special case when $R$ is given by the classical Yang-Baxter operator $R_q$ for some nonzero scalar $q$, the universal bialgebra $A(R_q)$ is the quantized coordinate ring of the matrix algebra $\mathcal O_q(M_n(\kk)),$ and its Hopf envelope is the quantized coordinate ring of the general linear group $\mathcal O_q(\GL_n(\kk))$. We describe all the twisted solutions of $R_q$ as follows.

\begin{introdethm}[\Cref{prop:twistR}]
Let $V$ be a finite-dimensional vector space over $\kk$ with basis $\{v_1,\ldots,v_n\}$ and let $R_q\in \textnormal{End}_\kk(V^{\otimes 2})$ be the classical Yang-Baxter operator. Denote by $\sigma$ the 2-cocycle given by some twisting pair of $A(R)$ and let $(R_q)^\sigma$ be the twisted solution. Then we have
\begin{enumerate}
    \item If $q=1$, then $(R_q)^\sigma(v_i\otimes v_j)=\sum_{1\leq k,l\leq n}\alpha^k_i\beta^l_jv_k\otimes v_l$.
    \item If $q=-1$, then $(R_q)^\sigma(v_i\otimes v_j)=\sum_{1\leq k,l\leq n}(-1)^{\delta_j^k}\,\delta^k_{\tau^{-1}(i)}\delta^l_{\tau(j)}\alpha^{\tau^{-1}(i)}_i(\alpha_{\tau(j)}^j)^{-1}v_k\otimes v_l$.

\item If $q\ne \pm1$, then 
    \begin{equation*}
(R_q)^\sigma(v_i\otimes v_j)=\left\{\begin{aligned}
&\sum_{1\leq k,l\leq n}\delta^{i}_{k}\delta^{j}_{l}\,\alpha_i^i\,(\alpha_l^l)^{-1}v_k\otimes v_l, & \textnormal{if}\quad  i<j\\
&\sum_{1\leq k,l\leq n}\delta^{i}_{k}\delta^{j}_{l}\,\alpha_i^i\,(\alpha_l^l)^{-1}\,qv_k\otimes v_l,
& \textnormal{if}\quad  i=j\\
&\sum_{1\leq k,l\leq n}\left(\delta^{i}_{k}\delta^{j}_{l}+\delta^{i}_{l}\delta_{j}^{k}(q-q^{-1})\right)\alpha_i^i\,(\alpha_l^l)^{-1}v_j\otimes v_l, 
& \quad \textnormal{if}\quad i>j,
\end{aligned}
\right.
\end{equation*}
\end{enumerate}
where $\delta_{j}^{i}$ is the Kronecker delta, and $(\alpha^i_j)$ and $(\beta^i_j)$ are $n\times n$ matrices inverse to each other. When $q=-1$, $(\alpha^i_j)$ is a generalized permutation matrix such that $\alpha^{i}_j\neq 0$ whenever $j=\tau(i)$ for some fixed $\tau\in S_n$, When $q\neq \pm 1$, $(a^i_j)$ is a diagonal matrix. 
\end{introdethm}

%%%%%%%%%%%%%%%%%%%%%%%%%%%%%%%%%%%%%%%%%
\noindent
{\bf Acknowledgements.} The authors thank Chelsea Walton for useful discussions that inspired this paper. The authors also thank the referees for their insightful comments. Nguyen was supported by the Naval Academy Research Council. Vashaw was partially supported by an Arthur K. Barton Superior Graduate Student Scholarship in Mathematics from Louisiana State University, NSF grants DMS-1901830 and DMS-2131243, and NSF Postdoctoral Fellowship DMS-2103272. Wang was partially supported by Simons collaboration grant \#688403 and AFOSR grant FA9550-22-1-0272. Part of this research work was done during Wang's visit to the Department of Mathematics at Rice University in October 2021. He is grateful for the first author’s invitation and wishes to thank Rice University for its hospitality.

%%%%%%%%%%%%%%%%%%%%%%%%%%%%%%%%%%%%%%%%%
\section{Zhang twists and 2-cocycle twists of Hopf algebras}
\label{sec:twists}

Throughout the paper, let $\kk$ be a base field with tensor product $\otimes$ taken over $\kk$ unless stated otherwise. All algebras are associative over $\kk$. We use the Sweedler notation for the comultiplication in a coalgebra $B$ such that $\Delta(h) = \sum h_1 \otimes h_2$ for any $h \in B$. 

We begin this section by proving fundamental properties of twisting pairs of a bialgebra as defined in \Cref{defn:twisting pair} and proving that they give rise to the twisting pairs of the corresponding Hopf envelope (\Cref{cor:lifting}). We then show that for a bialgebra or a Hopf algebra that satisfies our twisting conditions, its Zhang twist is again a bialgebra or a Hopf algebra (\Cref{zhang-twist-of-bialgebra}). 
In particular, we show that the construction of Hopf envelopes and that of Zhang twists are compatible (\Cref{f-iso}). Finally, we connect the notion of a Zhang twist of a graded Hopf algebra with that of a $2$-cocycle twist (\Cref{twist-and-2-cocycle}).

%%%%%%%%%%%%%%%%%%%%%%%%%%%%%%%%%%%%%%%%%
\subsection{Twisting pairs and liftings to Hopf envelopes}

We provide some basic properties of twisting pairs introduced in \Cref{defn:twisting pair}. As a consequence of \Cref{lem:twisting pair prop}, for the rest of the paper, whenever we refer to a twisting pair, it satisfies conditions (\textbf{P1})-(\textbf{P4}).

\begin{lemma}
\label{lem:twistpairinverse}
For a twisting pair $(\phi_1, \phi_2)$, the inverse pair $(\phi_1^{-1}, \phi_2^{-1})$ is again a twisting pair.
\end{lemma}
\begin{proof}
For a twisting pair $(\phi_1, \phi_2)$, it is straightforward to verify
\[  (\id\otimes \phi_1^{-1})\circ \Delta\overset{({\bf P1)}}{=}\Delta\circ \phi_1^{-1}, \quad (\phi_2^{-1}\otimes \id)\circ \Delta\overset{({\bf P1)}}{=}\Delta \circ \phi_2^{-1}, 
\text{ and }
\varepsilon \circ (\phi_1^{-1}\circ \phi_2^{-1}) \overset{({\bf P2)}}{=}\varepsilon. \]
Hence, the inverse pair $(\phi_1^{-1}, \phi_2^{-1})$ is again a twisting pair.
\end{proof}

\begin{lemma}\label{lem:twisting pair prop}
Let $(B,m,u,\Delta,\varepsilon)$ be a bialgebra. For any twisting pair $(\phi_1,\phi_2)$ of $B$, we have the following properties:
\begin{enumerate}
    \item[(\textbf{P3})] $\phi_1\circ \phi_2=\phi_2\circ \phi_1$, and
    \item[(\textbf{P4})] $(\phi_1\otimes \phi_2)\circ \Delta=\Delta$.
\end{enumerate}
\end{lemma}
\begin{proof}
We will use the properties (\textbf{P1})-(\textbf{P2})
%in \Cref{defn:twisting pair}
to prove (\textbf{P3})-(\textbf{P4}).
For (\textbf{P3}), we have
\begin{align*}
  \phi_1\circ \phi_2&=(\varepsilon\otimes \id)\circ \Delta\circ \phi_1\circ \phi_2\overset{({\bf P1})}{=}(\varepsilon\otimes \id)\circ(\id \otimes \phi_1)\circ \Delta\circ \phi_2\overset{({\bf P1})}{=}(\varepsilon\otimes \id)\circ(\id\otimes \phi_1)\circ (\phi_2 \otimes\id)\circ \Delta\\
  &=(\varepsilon\otimes \id)\circ(\phi_2\otimes \id)\circ (\id\otimes \phi_1)\circ \Delta \overset{({\bf P1})}{=} (\varepsilon\otimes \id)\circ(\phi_2\otimes \id)\circ \Delta\circ \phi_1 \overset{({\bf P1})}{=} (\varepsilon\otimes \id)\circ\Delta\circ \phi_2\circ \phi_1\\
  &=\phi_2\circ \phi_1. 
\end{align*}
For (\textbf{P4}), we first remark that $\phi_1=(\id \otimes \varepsilon)\circ \Delta\circ \phi_1 \overset{({\bf P1})}{=} (\id \otimes \varepsilon)\circ (\id \otimes \phi_1)\circ \Delta =(\id \otimes (\varepsilon\circ  \phi_1))\circ \Delta$. Hence, 
\begin{align}
\phi_1 = (\id \otimes (\varepsilon\circ  \phi_1))\circ \Delta. \label{eq:P4}
\end{align}
Similarly, we have $\phi_2=((\varepsilon\circ \phi_2)\otimes \id)\circ \Delta$, so that
\begin{align*} 
    \left((\varepsilon\circ \phi_1)\otimes (\varepsilon\circ \phi_2)\right)\circ \Delta &=\left((\varepsilon\circ \phi_1
\circ \phi_2)\otimes (\varepsilon\circ \phi_2
\circ \phi_1)\right)\circ (\phi_2^{-1}\otimes \phi_1^{-1})\circ \Delta \\
 \overset{({\bf P3})}&=\left((\varepsilon\circ \phi_1
\circ \phi_2)\otimes (\varepsilon\circ \phi_1
\circ \phi_2)\right)\circ (\phi_2^{-1}\otimes \phi_1^{-1})\circ \Delta \\ 
\overset{({\bf P1}),\; ({\bf P2})}&{=}(\varepsilon\otimes \varepsilon)\circ \Delta \circ \phi_1^{-1}\circ \phi_2^{-1}=
\varepsilon\circ \phi_1^{-1}\circ \phi_2^{-1} 
\overset{({\bf P2)}}{=} \varepsilon.
\end{align*}
Therefore,
\begin{align}
    \left((\varepsilon\circ \phi_1)\otimes (\varepsilon\circ \phi_2)\right)\circ \Delta & = \varepsilon.\label{eq:1.1.3}
\end{align}
Hence,
\begin{align*}
    (\phi_1\otimes \phi_2)\circ \Delta \overset{(\ref{eq:P4})}&=\left(\id \otimes(\varepsilon\circ \phi_1)\otimes (\varepsilon\circ \phi_2)\otimes \id \right)\circ (\Delta\otimes \Delta)\circ \Delta \\
    \overset{{\rm Coassociativity}}&{=}\left(\id \otimes(\varepsilon\circ \phi_1)\otimes (\varepsilon\circ \phi_2)\otimes \id \right)\circ (\id \otimes \Delta\otimes \id)\circ (\Delta\otimes \id)\circ  \Delta\\
    \overset{(\ref{eq:1.1.3})}&=(\id\otimes \varepsilon\otimes \id)\circ (\Delta\otimes \id)\circ  \Delta =\Delta. \qedhere
\end{align*}
\end{proof}

\begin{Cor}\label{eq:P1inverse}
The set of all twisting pairs of a bialgebra has a group structure under component-wise multiplication and inverse.
\end{Cor}

\begin{proof}
By \Cref{lem:twistpairinverse}, it is enough to show that the set of twisting pairs is closed under composition of pairs of morphisms.
Suppose $(\phi_1, \phi_2)$ and $(\varphi_1, \varphi_2)$ are two twisting pairs. We now show that $(\phi_1\circ \varphi_1,\phi_2\circ \varphi_2)$ is again a twisting pair. For (\textbf{P1}), we have 
\[
(\id \otimes (\phi_1\circ \varphi_1))\circ \Delta=(\id\otimes \phi_1)\circ(\id \otimes \varphi_1)\circ \Delta\overset{({\bf P1)}}{=}(\id\otimes \phi_1)\circ \Delta\circ \varphi_1\overset{({\bf P1)}}{=}\Delta\circ \phi_1\circ \varphi_1=\Delta\circ (\phi_1\circ \varphi_1).
\]
Similarly, we have $((\phi_2\circ \varphi_2)\otimes \id)\circ \Delta=\Delta\circ (\phi_2\circ \varphi_2)$.

For (\textbf{P2}), we can follow the argument for (\textbf{P3}) to get $\varphi_1\circ \phi_2=\phi_2\circ \varphi_1$. Then 
\[\varepsilon\circ \phi_1\circ \varphi_1\circ \phi_2\circ \varphi_2=\varepsilon\circ \phi_1\circ \phi_2\circ \varphi_1\circ \varphi_2\overset{({\bf P2)}}{=}\varepsilon\circ \varphi_1\circ \varphi_2\overset{({\bf P2)}}{=}\varepsilon. \qedhere \]
\end{proof}

Next, we provide the notion of winding endomorphisms (see e.g., \cite[\S 2.5]{BZ08} in the context of Hopf algebras) and show in \Cref{lem:winding} that for any twisting pair $(\phi_1,\phi_2)$ of a bialgebra $B$, $\phi_1$ and $\phi_2$ are uniquely determined by each other as winding endomorphisms. For any algebra homomorphism $\pi: B \to \kk$, let $\Xi^l[\pi]$ denote the \emph{left winding endomorphism} of $B$ defined by 
\[\Xi^l[\pi] = m \circ (\pi \otimes \id) \circ \Delta, \quad \text{that is,} \quad \Xi^l[\pi](h)=\sum \pi(h_1)h_2, \quad \text{for } h \in H.\]
Similarly, let $\Xi^r[\pi]$ denote the \emph{right winding endomorphism} of $B$ defined by
\[\Xi^r[\pi] = m \circ (\id \otimes \pi) \circ \Delta, \quad \text{that is,} \quad \Xi^r[\pi](h)=\sum h_1\pi(h_2), \quad \text{for } h \in H.\]
If in addition $B$ is a Hopf algebra with antipode $S$, one can check that the inverse of $\Xi^{l}[\pi]$ is $(\Xi^{l}[\pi])^{-1}=\Xi^{l}[\pi\circ S]$ making $\Xi^{l}[\pi]$ a winding automorphism. Analogously, $\Xi^r[\pi]$ is also a winding automorphism in the case of a Hopf algebra. The following result shows that twisting pairs of bialgebras (resp. Hopf algebras) are special pairs of winding endomorphisms (resp. automorphisms). 
\begin{lemma}
\label{lem:weak twisting pair}
For any algebra endomorphism $\phi$ of a bialgebra $B$, we have the following:
\begin{enumerate}
    \item The map $\phi$ satisfies $(\id \otimes \phi)\circ \Delta=\Delta\circ \phi$ if and only if $\phi$ is a right winding endomorphism of $B$. In this case, we have $\phi=\Xi^r[\varepsilon\circ \phi]$.
    \item The map $\phi$ satisfies $(\phi\otimes \id)\circ \Delta=\Delta\circ \phi$ if and only if $\phi$ is a left winding endomophism of $B$. In this case, we have $\phi=\Xi^l[\varepsilon\circ \phi]$.
\end{enumerate}
\end{lemma}
\begin{proof}
We only show (1) here, and the proof for (2) follows similarly by symmetry. We first assume that $(\id \otimes \phi)\circ \Delta=\Delta\circ \phi$. Then we have
\begin{align*}
    \phi=\id \circ \phi=(\id \otimes \varepsilon)\circ 
\Delta\circ \phi=(\id \otimes \varepsilon)\circ (\id \otimes \phi)\circ \Delta=(\id \otimes (\varepsilon\circ \phi))\circ \Delta=\Xi^r[\varepsilon\circ \phi], 
\end{align*}
showing that $\phi$ is a right winding endomorphism of $B$. Conversely, suppose $\phi=\Xi^r[\pi]$ is a right winding endomorphism of $B$ for some algebra homomorphism $\pi: B\to \kk$. Then we have 
\begin{align*}
    (\id \otimes\phi)\circ \Delta&\,=(\id \otimes \Xi^r[\pi])\circ \Delta=(\id \otimes \id \otimes \pi)\circ (\id \otimes \Delta)\circ \Delta=(\id \otimes \id \otimes \pi)\circ (\Delta\otimes \id)\circ \Delta\\
    &\,=\Delta\circ (\id \otimes \pi)\circ \Delta=\Delta\circ \Xi^r[\pi]=\Delta\circ \phi.\qedhere
\end{align*}
\end{proof}

\begin{remark}
Let $(\phi_1,\phi_2)$ be a pair of algebra \emph{endomorphisms}, not necessarily automorphisms, of a Hopf algebra $H$ satisfying (\textbf{P1})-(\textbf{P2}). By \Cref{lem:weak twisting pair}, $\phi_1$ and $\phi_2$ are winding endomorphisms of $H$; and hence they are automorphisms with inverses 
$$ 
\phi_1^{-1} = \Xi^r[ \varepsilon \circ \phi_1 \circ S] \qquad and \qquad 
\phi_2^{-1} = \Xi^l[\varepsilon \circ \phi_2 \circ S].
$$ 
Therefore, $(\phi_1, \phi_2)$ is  a twisting pair of $H$.
\end{remark}

\begin{lemma}
\label{lem:winding}
For any twisting pair $(\phi_1,\phi_2)$ of a bialgebra $B$, we have 
\[\phi_1^{\pm 1}=\Xi^r[\varepsilon\circ \phi_2^{\mp 1}]\quad \text{and}\quad \phi_2^{\pm 1}=\Xi^l[\varepsilon\circ \phi_1^{\mp 1}].\]
Moreover, for any algebra automorphism $\phi$ of $B$ satisfying $\Delta \circ \phi=(\id \otimes \phi)\circ \Delta$ (resp. $\Delta \circ \phi=(\phi\otimes \id)\circ \Delta$), the pair of maps $(\phi, \Xi^l[\varepsilon\circ \phi^{-1}])$ (resp. $(\Xi^r[\varepsilon\circ \phi^{-1}],\phi)$) is a twisting pair of $B$. 
\end{lemma}
\begin{proof}
It follows from (\textbf{P4}) and \Cref{eq:P1inverse} that 
\begin{align*}
    \phi_1^{\pm 1}=(\phi_1^{\pm 1}\otimes \varepsilon)\circ \Delta=(\id \otimes (\varepsilon\circ \phi_2^{\mp 1}))\circ (\phi_1^{\pm 1}\otimes \phi_2^{\pm 1})\circ \Delta=(\id \otimes (\varepsilon\circ \phi_2^{\mp 1}))\circ \Delta=\Xi^r[\varepsilon\circ \phi_2^{\mp 1}].
\end{align*}
One can argue for $\phi_2^{\pm 1}$ similarly. 

Now suppose $\phi$ is an algebra automorphism of $B$ satisfying $\Delta \circ \phi=(\id \otimes \phi)\circ \Delta$. We remark that in this case also $\Delta \circ \phi^{-1}=(\id \otimes \phi^{-1} )\circ \Delta$. Denote  $\phi'=\Xi^l[\varepsilon\circ \phi^{-1}]$. We check that $(\phi, \phi')$ is a twisting pair of $B$ by verifying the conditions in \Cref{defn:twisting pair}.

For (\textbf{P1}), we have
\begin{align*}
    \Delta\circ \phi'&=\Delta\circ ((\varepsilon\circ \phi^{-1})\otimes \id)\circ \Delta=((\varepsilon\circ \phi^{-1})\otimes \id\otimes \id)\circ (\id \otimes \Delta)\circ \Delta=((\varepsilon\circ \phi^{-1})\otimes \id\otimes \id)\circ (\Delta \otimes \id)\circ \Delta\\
    &=(((\varepsilon\circ \phi^{-1})\otimes \id)\circ \Delta)\otimes \id)\circ \Delta=(\phi'\otimes \id)\circ \Delta.
\end{align*}
And for (\textbf{P2}), we have
\begin{align*}
    \varepsilon\circ \phi\circ \phi'&=\varepsilon\circ \phi \circ ((\varepsilon\circ \phi^{-1})\otimes \id)\circ \Delta=((\varepsilon\circ \phi^{-1})\otimes \varepsilon)\circ (\id \otimes \phi)\circ \Delta=((\varepsilon\circ \phi^{-1})\otimes \varepsilon)\circ\Delta\circ \phi\\
    &=((\varepsilon\circ \phi^{-1})\otimes \id)\circ (\id\otimes \varepsilon)\circ\Delta\circ \phi=\varepsilon\circ \phi^{-1}\circ \phi=\varepsilon.
\end{align*}
A similar argument may be used for the case when $\Delta\circ \phi=(\phi\otimes \id)\circ \Delta$.
\end{proof}

\begin{thm}
\label{thm:twistingpair}
Let $(B,m,u,\Delta,\varepsilon)$ be a bialgebra. Then any twisting pair of $B$ is given by 
\[
\left(\Xi^r[\pi_1],\Xi^l[\pi_2]\right)
\]
where $\pi_1,\pi_2:B\to \kk$ are two algebra homomorphisms satisfying $(\pi_2\otimes \pi_1) \circ \Delta=(\pi_1\otimes \pi_2) \circ \Delta=\varepsilon$. Moreover, if $B$ is a Hopf algebra with antipode $S$, then we have $\pi_1=\pi_2\circ S$ and  $\pi_2=\pi_1\circ S$.
\end{thm}
\begin{proof}
According to \Cref{lem:weak twisting pair} and \Cref{lem:winding}, any twisting pair of $B$ is given by $(\Xi^r[\pi_1],\Xi^l[\pi_2])$ for two algebra homomorphisms $\pi_1,\pi_2: B\to \kk$. We now claim that (i) \textbf{(P2)} holds if and only if $(\pi_2\otimes \pi_1)\circ \Delta=\varepsilon$, and (ii) \textbf{(P4)} holds if and only if $(\pi_1\otimes \pi_2)\circ \Delta=\varepsilon$. For $b \in B$, we have
\begin{align*}
    (\varepsilon \circ \phi_1 \circ \phi_2) (b) &=\sum \varepsilon \circ \phi_1 (\pi_2(b_1) b_2)\\
    &= \sum \varepsilon(\pi_2(b_1) b_2 \pi_1 (b_3))\\
    &=\sum \pi_2 (b_1 \varepsilon(b_2)) \pi_1(b_2)\\
    &= \sum \pi_2(b_1) \pi_1(b_2)\\
    &=((\pi_2 \otimes \pi_1) \circ \Delta )(b).
\end{align*}
Claim (i) follows. To prove (ii), suppose on one hand that $(\pi_1 \otimes \pi_2) \circ \Delta = \varepsilon$. Then note that
\begin{align*}
    ((\phi_1 \otimes \phi_2) \circ \Delta)(b)&= \sum (\phi_1 \otimes \phi_2)(b_1 \otimes b_2)\\
    &= \sum (b_1 \pi_1(b_2) \otimes \pi_2 (b_3) b_4)\\
    &=\sum((b_1 \otimes b_4) \cdot (\pi_1(b_2) \otimes \pi_2(b_3))\\
    &=\sum ((b_1 \otimes b_3)\cdot \varepsilon(b_2))\\
    &=\sum b_1 \otimes \varepsilon(b_2) b_3\\
    &=\Delta(b)
\end{align*}
for all $b \in B$, which shows that \textbf{(P4)} holds. For the other direction, assume \textbf{(P4)} holds, in other words,
\begin{align*}
    \sum b_1 \pi_1(b_2) \otimes \pi_2(b_3) b_4 &=  \sum b_1 \otimes b_2
\end{align*}
for all $b \in B$. Applying $\varepsilon \otimes \varepsilon$ to both sides, we obtain
\begin{align*}
    \sum \varepsilon(b_1) \pi_1(b_2) \pi_2(b_3) \varepsilon (b_4) &=\varepsilon(b),
\end{align*}
from which it follows that 
\begin{align*}
    \sum \pi_1(b_1) \pi_2(b_2) &= \epsilon(b)
\end{align*}
for all $b \in B$, in other words, $(\pi_1\otimes \pi_2)\circ \Delta=\varepsilon$. This shows (ii). 

It remains to show that $\Xi^r[\pi_1]$ and $\Xi^l[\pi_2]$ are algebra automorphisms of $B$, which follows from $(\Xi^r[\pi_1])^{-1}=\Xi^r[\pi_2]$ and $(\Xi^l[\pi_2])^{-1}=\Xi^l[\pi_1]$.  

Furthermore, suppose $B$ is a Hopf algebra with antipode $S$. Then $(\pi_2\otimes \pi_1)\circ \Delta=\varepsilon$ if and only if $\pi_2*\pi_1=u\varepsilon$ in ${\rm End}_\kk(B)$ with respect to the convolution product $*$. Since $\pi_1\circ S$ is the two-sided inverse of $\pi_1$ with respect to $*$, we get $\pi_2=\pi_1\circ S$ and similarly $\pi_1=\pi_2\circ S$.
\end{proof}

\begin{remark} 
As a consequence of the previous result, if $B$ is a bialgebra satisfying the twisting conditions (\textbf{T1})-(\textbf{T2}), then any twisting pair of $B$ naturally preserves the algebra grading of $B$.
\end{remark}

\begin{Ex}
\label{lem:twistingAG}
Here we give some examples of twisting pairs of various Hopf algebras in view of \Cref{thm:twistingpair}. More examples will be considered in depth in \Cref{sec:Manin} and \Cref{section-twist-qybe}.
\begin{enumerate}
    \item Let $G$ be any algebraic group over $\kk$, and denote by $\mathcal O(G)$ its coordinate ring. Suppose $\pi: \mathcal O(G)\to \kk$ is any algebra homomorphism. Then, ${\rm ker}(\pi)$ is a maximal ideal of $\mathcal O(G)$ of codimension one which corresponds to a group element $g$ of $G$. Moreover, the winding automorphisms $\Xi^r[\pi]$ and $\Xi^l[\pi]$ are indeed algebra automorphisms of $\mathcal O (G)$ induced by the right and left translation on $G$ by $g$, which we denote by $r_g$ and $\ell_g$, respectively. One further checks that if there are two algebra homomorphisms $\pi_1,\pi_2: \mathcal O(G)\to \kk$, then $\pi_1=\pi_2\circ S$ if and only if the group elements in $G$ given by $\pi_1$ and $\pi_2$ are inverse to each other. Therefore, for any $g \in G$, $\left(r_g, \ell_{g^{-1}} \right)$ is a twisting pair for $\mathcal O(G)$, and moreover every twisting pair of $\mathcal O(G)$ has this form.
    \item Let $F\in \GL_n(\kk)$. The universal cosovereign Hopf algebra $H(F)$ introduced by Bichon in \cite{Bichon2001} is the algebra with generators $\mathbb U=(u_{ij})_{1\leq i,j\leq n}, \mathbb V=(v_{ij})_{1\leq i,j\leq n}$ and relations
\[\mathbb U\mathbb V^T=I_n=\mathbb V^T\mathbb U \qquad and \qquad \mathbb VF\mathbb U^TF^{-1}=I_n=F\mathbb U^TF^{-1}\mathbb V.\]
Here, we denote the matrix transpose by $(-)^T$. The Hopf algebra structure of $H(F)$ is given by
\begin{gather*}
    \Delta(u_{ij})=\sum_{1\leq k\leq n} u_{ik}\otimes u_{kj},\qquad \Delta(v_{ij})=\sum_{1\leq k\leq n} v_{ik}\otimes v_{kj},\\
    \varepsilon(u_{ij})=\delta_{ij}=\varepsilon(v_{ij}),\\
    S(\mathbb U) =\mathbb V^T,\qquad  S(\mathbb V)=F\mathbb U^TF^{-1}.
\end{gather*}
It is immediate that any algebra homomorphism $\pi: H(F)\to \kk$ is given by a pair of invertible matrices $A,B\in \GL_n(\kk)$ such that $\pi(\mathbb U)=A$ and $\pi(\mathbb V)=B$ satisfying 
\[
AB^T=I_n=B^TA \qquad \text{and} \qquad BFA^TF^{-1}=I_n=FA^TF^{-1}B.
\]
This yields that $B=(A^{-1})^T$ and $FA^T=A^TF$. As a consequence, every twisting pair $(\phi_1,\phi_2)$ of $H(F)$ is given by
\[
\phi_1(\mathbb U)=\mathbb UA,\quad \phi_1(\mathbb V)=\mathbb V (A^{-1})^T;\quad \phi_2(\mathbb U)=A^{-1}\mathbb U,\quad \phi_2(\mathbb V)=A^T\mathbb V,
\]
for some $A\in \GL_n(\kk)$ that commutes with $F^T$.
\item Let $H$ be any cocommutative Hopf algebra over an algebraically closed field $\kk$ of characteristic zero. By the Cartier-Gabriel-Kostant Theorem \cite[\S 5.6]{M1993}, $H\cong U(\mathfrak g)\#\kk[G]$ where the smash product is constructed from the action of the group algebra $\kk[G]$ on the universal enveloping algebra $U(\mathfrak{g})$, for some group $G$ and Lie algebra $\mathfrak g$. Any algebra homomorphism $\pi: U(\mathfrak g)\#\kk[G]\to \kk$ is given by a pair of characters $(\rho,\chi)$, where $\rho: \mathfrak g\to \kk$ and $\chi: G\to \kk^\times$ are associated with some one-dimensional representations of the Lie algebra $\mathfrak g$ and the group $G$, respectively. Moreover, we have $\pi(\mathfrak g)=\rho(\mathfrak g)$ and $\pi(G)=\chi(G)$, where $\rho$ is constant on $G$-orbits of $\mathfrak g$. Therefore, any twisting pair $(\phi_1,\phi_2)$ of $U(\mathfrak g)\#\kk[G]$ is given by
\[
\phi_1(x)=x+\rho(x),\quad \phi_1(g)=\chi(g)g;\quad \phi_2(x)=x-\rho(x),\quad \phi_2(g)=\chi(g)^{-1}g,
\]
for any $x\in \mathfrak g$ and any $g\in G$ for such a pair of characters $(\rho,\chi)$. One observes that for any twisting pair $(\phi_1,\phi_2)$ of a cocommutative Hopf algebra $H$, we always have $\phi_1\circ \phi_2=\id$.
\end{enumerate}
\end{Ex}

Next, we aim to lift the twisting of a bialgebra to that of a Hopf algebra through the construction of a Hopf envelope. To that end, we examine here how a twisting pair of $B$ can be lifted to a twisting pair of its Hopf envelope. It is well-known that the forgetful functor from the category of Hopf algebras to the category of bialgebras has a left adjoint. That is, for a bialgebra $B$, there exists a \emph{Hopf envelope} of $B$ which is a Hopf algebra, denoted by $\mathcal H(B)$, together with a bialgebra map $i_B: B\to \mathcal H(B)$ satisfying the following universal property: for any Hopf algebra $H$ together with a bialgebra map $\phi: B\to H$, there exists a unique Hopf algebra map $f: \mathcal H(B)\to H$ such that the following diagram commutes:
\[
\xymatrix{
B\ar[r]^-{i_B}\ar[dr]_-{\phi} & \mathcal H(B)\ar[d]^-{f}\\
& H.
}
\]

We use the explicit construction of the Hopf envelope $\mathcal H(B)$ given by Takeuchi in \cite{Takeuchi1971}. First, we consider the bialgebra $B=\kk\langle V\rangle/(R)$, where $V$ is a $\kk$-vector space. Here, the comultiplication $\Delta$ and counit $\varepsilon$ of $B$ are induced from the free algebra via $\Delta: \kk\langle V\rangle \to \kk\langle V\rangle\otimes \kk\langle V\rangle$ and $\varepsilon: \kk\langle V\rangle \to \kk$ as algebra maps. In particular, $R$ is a bi-ideal such that 
\[\Delta(R)\subseteq R\otimes \kk\langle V\rangle+\kk\langle V\rangle \otimes R \qquad \text{and} \qquad \varepsilon(R)=0.\] 
We denote infinitely many copies of the generating space $V$ as $\{V^{(i)}=V\}_{i\ge 0}$ and consider the free algebra 
\begin{equation}\label{eq:freeT}
T:=\kk\langle\, \oplus_{i\ge 0} V^{(i)}\rangle.
\end{equation}
Let $S$ be the anti-algebra map on $T$ with $S(V^{(i)})=V^{(i+1)}$ for any $i\ge 0$. Both algebra maps $\Delta: \kk\langle V\rangle \to \kk\langle V\rangle \otimes \kk\langle V\rangle$ and $\varepsilon: \kk\langle V\rangle\to \kk$ extend uniquely to $T$ as algebra maps via $S_{(12)}(S\otimes S)\circ \Delta=\Delta \circ S$ and $\varepsilon\circ S=\varepsilon$, which we still denote by $\Delta: T\to T\otimes T$ and $\varepsilon: T\to \kk$ accordingly. Here, $S_{(12)}$ is the flip map on the 2-fold tensor product. One can check that the Hopf envelope of $B$ has a presentation
\[
\mathcal H(B)=T/W,   
\]
where the ideal $W$ is generated by 
\begin{equation}\label{eq:RHB}
 S^i(R), \quad (m \circ (\id \otimes S) \circ \Delta-u \circ \varepsilon)(V^{(i)}), \quad \text{and} \ (m \circ (S\otimes \id) \circ \Delta-u \circ \varepsilon)(V^{(i)}),\quad \text{for all } i\ge 0.   
\end{equation}
One can check that $W$ is a Hopf ideal of $T$, and so the Hopf algebra structure maps $\Delta$, $\varepsilon$, and $S$ of $T$ give a Hopf algebra structure on $T/W$. Finally, the natural bialgebra map $i_B: B\to \mathcal H(B)$ is given by the natural embedding $\kk\langle V\rangle \hookrightarrow T$ by identifying $V=V^{(0)}$.

\begin{lemma}
\label{t1-t3-hopf-env} 
Let $B$ be a bialgebra satisfying the twisting conditions (\textbf{T1})-(\textbf{T2}). The following hold:
\begin{enumerate}
    \item Its Hopf envelope $\mathcal H(B)$ also satisfies the twisting conditions (\textbf{T1})-(\textbf{T2}).
    \item The natural bialgebra map $i_B: B\to \mathcal H(B)$ preserves the algebra gradings. 
    \item Let $S$ be the antipode of $\mathcal H(B)$. Then we have $S(\mathcal H(B)_n)\subseteq \mathcal H(B)_{-n}$ for any $n\in \mathbb Z$.
\end{enumerate}
\end{lemma}
\begin{proof} 
Suppose $B$ satisfies (\textbf{T1}) and (\textbf{T2}). In particular, $B$ is $\mathbb Z$-graded. Then, in the presentation of $B=\kk\langle V\rangle/(R)$, we may choose $V$ to be a graded vector space and $R$ a homogeneous ideal in the free algebra $\kk\langle V\rangle$ with internal grading.
We construct the Hopf envelope $\mathcal H(B)$ as above and set ${\rm deg}\, S(v_i)=-{\rm deg}(v_i)$ inductively for all homogeneous $v_i\in V^{(i)}$ and $i\ge 0$. Then it follows that $T=\bigoplus_{n\in \mathbb Z} T_n$ in \eqref{eq:freeT} is again graded by internal grading (also called Adam's grading, induced from the grading of $B$). Moreover, we may check that the comultiplication $\Delta: T\to T\otimes T$ satisfies $\Delta(T_n)\subseteq T_n\otimes T_n$. As a consequence, $W$ is generated by homogeneous elements in $T$. Hence $\mathcal H(B)=T/W$ satisfies all the requirements in the statement. 
\end{proof} 

As a consequence, we introduce an additional twisting condition \textbf{(T3)} for any Hopf algebra.  

\begin{Cor}\label{cor:T3}
Let $H=\bigoplus_{n\in \mathbb Z} H_n$ be any Hopf algebra satisfying \textbf{(T1)}-\textbf{(T2)}. Then $H$ satisfies:
\begin{enumerate}
    \item[\textbf{(T3)}] $S(H_n)\subseteq H_{-n}$ for any $n\in \mathbb Z$.
\end{enumerate}
\end{Cor}
\begin{proof}
If $H$ is a Hopf algebra satisfying (\textbf{T1}) and (\textbf{T2}), then its Hopf envelope satisfies (\textbf{T1})-(\textbf{T3}) by \Cref{t1-t3-hopf-env}. By \Cref{t1-t3-hopf-env} (2) and the universal property,  $\mc{H}(H)$ is isomorphic to $H$ as bialgebras via $i_H$ preserving the algebra grading. Hence any Hopf algebra satisfying (\textbf{T1}) and (\textbf{T2}) also satisfies (\textbf{T3}).
\end{proof}

\begin{Ex}
\label{ex:H(e)}
For any $n$-dimensional vector space $V$ over $\kk$ and any integer $m \ge 2$, the universal quantum group $\mathcal H(e)$, defined for a preregular form $e: V^{\otimes m} \to \kk$ with generators $(a_{ij})_{1\leq i, j\leq n}$, $(b_{ij})_{1\leq i, j\leq n}$ and $D^{\pm 1}$, was studied in \cite{BDV13,Chirvasitu-Walton-Wang2019}. 
By setting ${\rm deg}(a_{ij})=1$, ${\rm deg}(b_{ij})=-1$ and ${\rm deg}(D^{\pm 1})=\pm n$, it is straightforward to check that $\mathcal H(e)$ satisfies the twisting conditions (\textbf{T1})-(\textbf{T3}). In later sections, we study in detail some other classes of (universal) quantum groups that also satisfy (\textbf{T1})-(\textbf{T3}).
\end{Ex}

We now show that any twisting pair of a bialgebra gives rise to a twisting pair of its Hopf envelope. 

\begin{proposition}
\label{cor:lifting}
Any twisting pair of a bialgebra $B$ can be extended uniquely to a twisting pair of its Hopf envelope $\mathcal H(B)$. Moreover, any twisting pair of $\mathcal H(B)$ is obtained from some twisting pair of $B$ in such a way.
\end{proposition}
\begin{proof}
In view of \Cref{thm:twistingpair}, it suffices to show that any two algebra homomorphisms $\pi_1,\pi_2: B\to \kk$ satisfying $(\pi_2\otimes \pi_1)\circ \Delta=(\pi_1\otimes \pi_2)\circ \Delta=\varepsilon$ can be uniquely extended to algebra homomorphisms $\mathcal H(\pi_1),\mathcal H(\pi_2): \mathcal H(B)\to \kk$ satisfying $(\mathcal H(\pi_2)\otimes \mathcal H(\pi_1))\circ \Delta=(\mathcal H(\pi_1)\otimes \mathcal H(\pi_2))\circ \Delta=\varepsilon$. 

Recall the explicit construction of $\mc{H}(B)$ given above as $\mathcal H(B)=T/W$, where $B=\kk\langle V\rangle/(R)$, $T=\kk\langle \oplus_{i\ge 0} V^{(i)}\rangle$, and $W$ is an ideal generated by \eqref{eq:RHB}. In particular, the antipode of $\mathcal H(B)$ is induced by the anti-algebra map $S$ on $T$ given by $S(V^{(i)})=V^{(i+1)}$ for all $i\ge 0$.  

For $j \in\{1,2\}$, we first lift $\pi_j: B\to \kk$ to an algebra homomorphism $\widetilde{\pi_j}: \kk\langle V\rangle\to \kk$. Note that any lifting of $\pi_j$ from $B$ to $\mathcal H(B)$ is determined by the lifting of $\widetilde{\pi_j}$ from $\kk\langle V\rangle$ to $T$. The latter lifting is unique if we require that 
\begin{align}\label{eq:ext}
  \widetilde{\pi_j}=\widetilde{\pi_{j+1}}\circ S, 
\end{align} 
where the subscripts for $\pi_j$'s are indexed as $1,2$ (modulo $2$). We still denote such a lifting by $\widetilde{\pi_j}: T\to \kk$. 

It remains to show that $\widetilde{\pi_j}(W)=0$ in view of  \eqref{eq:RHB}.  Since the relation space $R\subseteq \kk\langle V\rangle$ vanishes under $\widetilde{\pi_j}$, one can show inductively that $\widetilde{\pi_j}(S^i(R))=0$ for all $i\ge 0$. Moreover, we show inductively that 
\[\widetilde{\pi_j}\left((m \circ (\id \otimes S) \circ \Delta-u \circ \varepsilon)(V^{(i)})\right)=\widetilde{\pi_j} \left((m\circ (S\otimes \id) \circ \Delta-u \circ \varepsilon)(V^{(i)})\right)=0\] 
for all $i\ge 0$. When $i=0$, we have
\[
\begin{aligned}
    \widetilde{\pi_j}\left((m \circ (\id \otimes S) \circ \Delta-u \circ \varepsilon)(V)\right)&=(m \circ (\widetilde{\pi_j}\otimes (\widetilde{\pi_j} \circ S)) \circ \Delta-\widetilde{\pi_j}\circ u \circ \varepsilon)(V)\\
    &=(m \circ (\widetilde{\pi_j}\otimes \widetilde{\pi_{j+1}}) \circ \Delta-\pi_j \circ u \circ \varepsilon)(V)\\
    &=(m \circ (\widetilde{\pi_j}\otimes \widetilde{\pi_{j+1}}) \circ \Delta-u \circ \varepsilon)(V)\\
    &=(m \circ (\pi_j\otimes\pi_{j+1}) \circ \Delta-u \circ \varepsilon)(V)\\
    &=(u\circ \varepsilon-u \circ \varepsilon)(V) = 0.
\end{aligned}
\]
In the second to last equation above, we use our assumption $(\pi_j\otimes \pi_{j+1})\circ \Delta=\varepsilon$. By induction, assume the statements hold for $i$, we have
\begin{align*}
    \widetilde{\pi_j}\left((m \circ (\id \otimes S) \circ \Delta-u \circ \varepsilon)(V^{(i+1)})\right)&=\widetilde{\pi_j} \left( (m\circ (\id \otimes S) \circ \Delta-u \circ \varepsilon)(S(V^{(i)}))\right)\\
    &=\widetilde{\pi_j}\left( (m \circ (\id \otimes S) \circ (S\otimes S) \circ S_{(12)} \circ \Delta-u \circ \varepsilon \circ S)(V^{(i)})\right)\\
    &=\widetilde{\pi_j}\left((m \circ (S\otimes S) \circ S_{(12)} \circ (S\otimes \id) \circ \Delta-S \circ u \circ \varepsilon)(V^{(i)})\right)\\
      &=(\widetilde{\pi_j} \circ S) \left((m \circ (S\otimes \id) \circ \Delta-u \circ \varepsilon)(V^{(i)})\right)\\
    &=\widetilde{\pi_{j+1}}\left((m \circ (S\otimes \id) \circ \Delta-u \circ \varepsilon)(V^{(i)}) \right)=0.
\end{align*}
Similarly, one can show that $\widetilde{\pi_j}\left((m \circ (S\otimes \id) \circ \Delta-u \circ \varepsilon)(V^{(i+1)})\right)=0$. Hence $\widetilde{\pi_j}(W)=0$ and it yields a well-defined algebra homomorphism $\mathcal H(\pi_j): \mathcal H(B)=T/W\to \kk$. By our construction, it is clear that $(\mathcal H(\pi_1),\mathcal H(\pi_2))$ are unique extensions of $(\pi_1,\pi_2)$ satisfying $(\mathcal H(\pi_2)\otimes \mathcal H(\pi_1))\circ \Delta=(\mathcal H(\pi_1)\otimes \mathcal H(\pi_2))\circ \Delta=\varepsilon$ and every such pair of $\mathcal H(B)$ arises in this way. 
\end{proof}

\begin{remark}\label{remark:twisting}
We wish to consider the Zhang twist of $\mathcal H(B)$ by $\mathcal H(\phi_1)\circ \mathcal H(\phi_2)$ in \Cref{thm:Zhang and 2-cocycle twists}. Before doing so, we remark the following from \Cref{cor:lifting}:
\begin{enumerate}
    \item For the twisting pair $(\mathcal H(\phi_1),\mathcal H(\phi_2))$ of $\mathcal H(B)$, the composition $\mathcal H(\phi_1\circ \phi_2)=\mathcal H(\phi_1)\circ \mathcal H(\phi_2)$ is a Hopf algebra automorphism of $\mathcal H(B)$.
    \item If $B$ satisfies the twisting conditions, then its Hopf envelope $\mathcal H( B)$ also satisfies the twisting conditions. So any twisting pair $(\phi_1,\phi_2)$ of $B$ extends uniquely to $\mathcal H( B)$ preserving its algebra grading.
\end{enumerate}
\end{remark}

%%%%%%%%%%%%%%%%%%%%%%%%%%%%%%%%%%%%%%%%%
\subsection{Zhang twists and liftings to Hopf envelopes}

The notion of a twist of a $\mathbb Z$-graded algebra was first introduced in \cite[Section 8]{ATV1991}. Using the more general notion of twisting systems, Zhang proved in \cite{Zhang1996} that for two $\mathbb N$-graded algebras $A$ and $B$ generated in degree one, $A$ is isomorphic to a twisted algebra of $B$ if and only if the graded module categories of $A$ and $B$ are equivalent. 

\begin{defn} 
\label{defn:Zhang-twist} 
Let $A$ be a $\mathbb {Z}$-graded algebra and $\phi$ be a graded automorphism of $A$. The \emph{right Zhang twist} $A^\phi$ of $A$ by $\phi$ is the algebra that coincides with $A$ as a graded $\kk$-vector space, with the twisted (or deformed) multiplication 
\[r *_\phi s = r \phi^{|r|}(s),
\,\, \text{for any homogeneous elements } r,s \in A. \]
Here we denote the grading degree of $r$ by $|r|$. Similarly, the \emph{left Zhang twist} $\prescript{\phi}{}{A}$ of $A$ by $\phi$ is defined with the twisted  multiplication 
\[r *_\phi s = \phi^{|s|}(r) s,
\,\, \text{for any homogeneous elements } r,s \in A. \]
When the context is clear, we may omit the subscript in the twisted multiplication $ *_\phi$. 
\end{defn} 

For $A$ and $\phi$ as in \Cref{defn:Zhang-twist}, the left and right Zhang twist of $A$ by $\phi$ is again a graded associative algebra \cite[Proposition 4.2]{Zhang1996}. Furthermore, there is a natural isomorphism $A^\phi \cong \prescript{\phi^{-1}}{}{A}$ given by the identity on sets \cite[Lemma 4.1]{Lecoutre-Sierra2019}. Zhang twists may be defined in the more general setting of an algebra graded by a semigroup by considering twisting systems \cite{Zhang1996}. In the case of a $\mathbb{Z}$-graded algebra, such a twisting system arises from a single algebra automorphism of the associated $\mathbb{Z}$-algebra \cite[Proposition 4.2]{Sierra2011}.

In this paper, we consider Zhang twists on bialgebras (resp.~Hopf algebras), but one should note that Zhang twists only deform the algebra structure. With this in mind, we let $B$ be a $\mathbb Z$-graded bialgebra (resp.~Hopf algebra). We say that an endomorphism $\phi\in {\rm End}_\kk(B)$ is a \emph{graded bialgebra automorphism} (resp.~\emph{graded Hopf algebra automorphism}) of $B$ if it is a bialgebra (resp. Hopf algebra) automorphism of $B$ while preserving the grading of $B$, that is, $\phi(B_n)\subseteq B_n$ for all $n\in \mathbb Z$. In what follows, we prove that the right Zhang twist $B^\phi$ of a graded bialgebra
(resp.~Hopf algebra) $B$ is again a graded bialgebra (resp.~Hopf algebra) under the twisting conditions (\textbf{T1})-(\textbf{T2}). Similar results can be proved for the left Zhang twist $^\phi B$ by using the fact that $\prescript{\phi}{}{B} \cong B^{\phi^{-1}}$.

\begin{proposition}
\label{zhang-twist-of-bialgebra}
Suppose $B$ is a bialgebra satisfying the twisting conditions (\textbf{T1})-(\textbf{T2}). Then we have the following.
\begin{enumerate}
 \item For any graded bialgebra automorphism $\phi$ of $B$, the Zhang twist $B^\phi$ together with the coalgebra structure of $B$ is again a bialgebra satisfying the twisting conditions (\textbf{T1})-(\textbf{T2}). 
 \item If in addition $B$ is a Hopf algebra with antipode $S$ and $\phi$ is any graded Hopf algebra automorphism of $B$, then $B^\phi$ is again a Hopf algebra satisfying the twisting conditions (\textbf{T1})-(\textbf{T2}), whose antipode is given by $S^\phi(r)=\phi^{-|r|}S(r)$, for all homogeneous elements $r\in B^{\phi}$.
\end{enumerate}
\end{proposition}

\begin{proof}
(1): By definition, $B^\phi$ is a graded algebra with the natural grading $B^\phi = \bigoplus_{i \in \mathbb{Z}} B_i$. Endow $B^\phi$ with the same coalgebra structure as $B$. To show that $B^\phi$ is a bialgebra, it suffices to show that $\Delta$ and $\varepsilon$ are algebra maps with respect to the deformed multiplication $*$.
Since $\phi$ is a bialgebra automorphism of $B$ preserving its degrees, we have 
\begin{align}
\label{eq:biauto}
    (\phi\otimes \phi) \circ \Delta =\Delta \circ \phi \qquad
    \textnormal{and} \qquad \varepsilon \circ \phi =\varepsilon. 
\end{align}
Thus, for any homogeneous elements $r,s \in B$, 
\begin{align*}
 \Delta(r*s) &= \Delta(r\phi^{|r|}(s)) = \Delta(r)\Delta(\phi^{|r|}(s)) = \sum r_1\phi^{|r|}(s)_1\otimes r_2\phi^{|r|}(s)_2 \\
 \overset{\eqref{eq:biauto},\, \textnormal{(\textbf{T2})}}&{=} \sum r_1\phi^{|r_1|}(s_1)\otimes r_2\phi^{|r_2|}(s_2) = \sum r_1 * s_1 \otimes r_2 * s_2 = \Delta(r)*\Delta(s),
\end{align*}
and
\[
 \varepsilon(r*s)= \varepsilon(r\phi^{|r|}(s))= \varepsilon(r)\varepsilon(\phi^{|r|}(s))
 \overset{\eqref{eq:biauto}}{=}\varepsilon(r)\varepsilon(s)=\varepsilon(r)*\varepsilon(s).
\]
Therefore, $B^\phi$ is a bialgebra as desired and still satisfies the twisting conditions (\textbf{T1})-(\textbf{T2}).

(2): Suppose $B$ is a Hopf algebra satisfying the twisting conditions. By (1), $B^{\phi}$ is a bialgebra satisfying the twisting conditions (\textbf{T1})-(\textbf{T2}). It remains to show that $B^{\phi}$ has an antipode.  
We define a map $S^\phi: B^{\phi} \to B^{\phi}$ by $S^\phi(r)=\phi^{-|r|}(S(r))$ for all homogeneous element $r\in B^{\phi}$, where $S$ is the original antipode map of $B$. Observe that $S^\phi$ is $\kk$-linear. For any $r \in B^\phi$, we have
\begin{align*}
\sum S^\phi(r_1)*r_2 &= \sum \phi^{-|r_1|}(S(r_1))*r_2 = \sum \phi^{-|r_1|}(S(r_1))\phi^{|S(r_1)|}(r_2) \\&= \sum \phi^{|S(r_1)|}(\phi^{-|S(r_1)|} \circ \phi^{-|r_1|}(S(r_1))r_2)
 \overset{\textnormal{(\textbf{T3})}}{=} \sum\phi^{|S(r_1)|}(\varepsilon(r)) =\varepsilon(r), 
\end{align*}
and
\[
\sum r_1*S^\phi(r_2)= \sum r_1*\phi^{-|r_2|}(S(r_2))= \sum r_1\phi^{|r_1|}(\phi^{-|r_2|}(S(r_2))) \overset{\textnormal{(\textbf{T2})}}{=}  \sum r_1S(r_2) =\varepsilon(r).
\]
Hence, $S^\phi$ is the antipode of $B^{\phi}$, so that $(B^{\phi},\, *\,, \Delta,\, S^\phi)$ is a Hopf algebra.   
\end{proof}

Let $H$ be a Hopf algebra satisfying the twisting conditions (\textbf{T1})-(\textbf{T2}). We denote by ${\rm Gr}$-$H$ the category of graded left $H$-modules.  
One can check that ${\rm Gr}$-$H$ is a monoidal category, with the \emph{Hadamard product $\overline{\otimes}$} of two graded $H$-modules $M=\bigoplus_{n\in \mathbb Z}M_n$ and $N=\bigoplus_{n\in \mathbb Z}N_n$ defined by
\[
\left(M\, \overline{\otimes}\, N\right)_n:=M_n\otimes N_n,
\]
where the left $H$-module structure is given via the comultiplication $\Delta$. The unit object in ${\rm Gr}$-$H$ is $\mathbb{1}=\bigoplus_{n\in \mathbb Z} \kk$, a direct sum of copies of $\kk$. In particular, if $M$ is locally finite, then $M$ has a dual given by
\[
M^*=\bigoplus_{n\in \mathbb Z}(M_n)^*=\bigoplus_{n\in \mathbb Z}  {\rm Hom}_{\kk}(M_n,\kk),
\]
where the left $H$-action on the $\kk$-dual $(M_n)^*$ is given by $(h \cdot f)(x)=f(S(h) \cdot x)$, for any $x \in M_{n+m}$, $f\in (M_n)^*$ and $h\in H_m$. Moreover, the evaluation map ${\rm ev}: M^*\otimes M\to \mathbb{1}$ and coevaluation map ${\rm coev}: \mathbb{1}\to M\otimes M^*$ are defined component-wise using natural maps $(M_n)^*\otimes M_n\to \kk$ and $\kk\to M_n\otimes (M_n)^*$. See \cite{EGNO} for more background on tensor categories. 

For any graded Hopf algebra automorphism $\phi$ of $H$, by a left-version of \Cref{zhang-twist-of-bialgebra} (2) the left Zhang twist $\!^\phi H$ is again a Hopf algebra satisfying the twisting conditions (\textbf{T1})-(\textbf{T2}). Its graded module category ${\rm Gr}$-$\,^\phi H$ is a monoidal category as well. Our next result extends the well-known equivalence of graded module categories between Zhang twists, given by Zhang in \cite{Zhang1996}, to their monoidal versions. 

\begin{proposition}
Let $H$ be a Hopf algebra satisfying the twisting conditions (\textbf{T1})-(\textbf{T2}). For any graded Hopf algebra automorphism $\phi$ of $H$, the monoidal categories ${\rm Gr}$-$H$ and ${\rm Gr}$-$\,^\phi H$ are equivalent.
\end{proposition}

\begin{proof}
By \cite[Corollary 8.5]{ATV1991}, ${\rm Gr}$-$H$ and ${\rm Gr}$-$\,^\phi H$ are equivalent as graded module categories via the equivalence $F: {\rm Gr}\text{-}H \to {\rm Gr}\text{-}\,^\phi H$ which is defined in the following way. Let $M=\bigoplus_{n\in \mathbb Z}M_n$ be a graded $H$-module. Then $F(M):=M$ as a graded vector space, and the new twisted $H$-action is given by $h\cdot_\phi m=\phi^{|m|}(h) \cdot m$, for any $m\in M$ and $h\in H$. It remains to verify that $F$ preserves the monoidal structure, that is, $F(M\overline{\otimes} N)\cong F(M)\overline{\otimes} F(N)$ in ${\rm Gr}\text{-}\,^\phi H$ for two graded $H$-modules $M$ and $N$. For any $m\in M_t$, $n\in N_t$, where $t \in \mathbb Z$, and $h\in H$, we have
\[
\begin{aligned}
 h\cdot_\phi(m\otimes n) &= \phi^{t}(h)\cdot (m\otimes n) = \Delta(\phi^{t}(h))\cdot (m\otimes n) = (\phi^{t}\otimes \phi^t)(\Delta(h)) \cdot (m\otimes n)\\
 &= \sum (\phi^t(h_1) \cdot m)\otimes (\phi^t(h_2) \cdot n) = \sum ( h_1\cdot_\phi m)\otimes ( h_2\cdot_\phi n).
\end{aligned}
\]
Then it is clear to see the equivalence functor $F: {\rm Gr}\text{-}H \to {\rm Gr}\text{-}\,^\phi H$ extends to a monoidal equivalence. 
\end{proof}

Let $i_B: B \to \mc{H}(B)$ be the natural bialgebra map. For any graded bialgebra automorphism $\psi$ of $B$,  the universal property of the Hopf envelope implies that there is a unique graded Hopf algebra automorphism of $\mathcal H(B)$, denoted by $\mathcal H(\psi)$, such that the following diagram commutes:
\begin{equation}\label{eq:7a}
\xymatrix{
B\ar[r]^-{i_B}\ar[d]_-{\psi} & \mathcal H(B)\ar[d]^-{\mathcal H(\psi)}\\
B\ar[r]^-{i_B} & \mathcal H(B).
}
\end{equation}
As a consequence, we can show that  
\begin{equation}\label{eq:liftingbiauto}
    \mathcal H(-): {\rm Aut}_{\text{grbi}}(B)\to {\rm Aut}_{\text{grHopf}}(\mathcal H(B))
\end{equation} 
is a group homomorphism from the group of graded bialgebra automorphisms of $B$ to the group of graded Hopf algebra automorphisms of $\mathcal{H}(B)$. Since $\psi$ and $\mathcal H(\psi)$ are both graded automorphisms of their respective algebras, we may consider the Zhang twists $B^\psi$ and $\mc{H}(B)^{\mathcal H(\psi)}$. By  \Cref{zhang-twist-of-bialgebra} and \Cref{t1-t3-hopf-env}, $B^\psi$ is a bialgebra and $\mc H(B)^{\mc H(\psi)}$ is a Hopf algebra, respectively.

\begin{lemma}
\label{i-map-both}
There exists a bialgebra map $i_\psi: B^\psi \to \mc{H}(B)^{\mc{H}(\psi)}$ that is equal to $i_B: B \to \mc{H}(B)$ as a map of vector spaces.
\end{lemma}
\begin{proof}
By our constructions, it is clear that $B^\psi=B$ and  $\mc{H}(B)^{\mathcal H(\psi)}=\mc{H}(B)$ as coalgebras. Thus the original bialgebra map $i_B: B\to \mc{H}(B)$ induces a coalgebra map $i_{\psi}:B^\psi \to \mc{H}(B)^{\mathcal H(\psi)}$ preserving the gradings. It remains to be shown that $i_{\psi}$ is also a graded algebra map, which follows from  
\begin{align*}
    i_{\psi}(x*_\psi y)=i_B(x\psi^{|x|}(y))\overset{\eqref{eq:7a}}{=}i_B(x)i_B(\psi^{|x|}(y))=i_B(x)\mathcal H(\psi)^{|i_B(x)|}(i_B(y))=i_{\psi}(x)*_{\mathcal H(\psi)} i_{\psi}(y),
\end{align*}
for homogeneous elements $x, y \in B^\psi$, and the fact that $i_B$ preserves the identity. 
\end{proof}

Our next goal is to show that $\mc{H}(B)^{\mc{H}(\psi)}$ is the Hopf envelope of $B^\psi$. For simplicity, we write 
\[\theta:=\mc{H}(\psi): \mc{H}(B)\to \mc{H}(B).\]
Denote the Hopf envelope of $B^\psi$ by $\mc{H}(B^\psi)$, together with the bialgebra map $i_{B^{\psi}}: B^\psi \to \mc{H}(B^\psi)$. By \Cref{i-map-both} and the universal property of $\mc{H}(B^\psi)$, we have a unique Hopf algebra map $f: \mc{H}(B^\psi) \to \mc{H}(B)^\theta$ making the following diagram commute:
\begin{equation}
\label{eq:hopfenvelope}
\xymatrix{
B^\psi \ar[r]^-{i_{B^\psi}} \ar[dr]_-{i_\psi} & \mc{H}\left( B^\psi\right) \ar[d]^-f\\
& \mc{H}\left( B\right)^\theta.
}
\end{equation}
Note that $\psi$ and $\psi^{-1}: B \to B$ are also graded algebra automorphisms of $B^\psi$ since 
\begin{align*}
\psi (x *_\psi y)&= \psi( x\psi^{|x|} (y))= \psi(x)\psi^{|x|+1} (y)= \psi(x) \psi^{|\psi(x)|+1}(y)= \psi(x) *_\psi \psi(y),
\end{align*}
for any homogeneous elements $x,y$ of $B$. Using this, we obtain from the universal property of $\mc{H}(B^\psi)$ a Hopf automorphism $\hat \theta:=\mathcal H(\psi^{-1}): \mc{H}(B^\psi) \to \mc{H}(B^\psi)$ making the following diagram commute:
\begin{equation}
    \label{eq:longtriangle}
\xymatrix{
B^\psi \ar[r]^-{\psi^{-1}} \ar[drr]_-{i_{B^\psi}} & B^\psi \ar[r]^-{i_{B^\psi}} & \mc{H} \left( B^\psi\right) \\
&& \mc{H}\left( B^\psi\right) \ar[u]_-{\hat\theta}
.}
\end{equation}

\begin{lemma}
\label{fs-and-thetas} 
Let $B, \psi, \theta, \hat \theta, f$ be as in the above discussion. As Hopf algebra maps $\mc{H}(B^\psi) \to \mc{H}(B)^\theta$, we have $\theta \circ f \circ \hat \theta = f$. 
\end{lemma}

\begin{proof}
The following diagram commutes by the definition and universal property of Hopf envelopes:
\[
\xymatrix{
B^\psi \ar[r]^-{\psi^{-1}} \ar[d]_-{i_{B^\psi}} & B^\psi \ar[r]^-\psi \ar[d]_-{i_{B^\psi}} \ar[rd]^-{i_\psi} & B^\psi \ar[rd]^-{i_\psi} \\
\mc{H} \left( B^\psi\right) \ar[r]_-{\hat\theta} & \mc{H} \left( B^\psi\right) \ar[r]_-f & \mc{H}\left( B\right)^\theta \ar[r]_-\theta & \mc{H} \left( B\right)^\theta.}
\]
The triangle and the leftmost square follow from \eqref{eq:hopfenvelope} and \eqref{eq:longtriangle}, respectively.
The rightmost square commutes, since it is equal (as a diagram of vector spaces and linear maps) to \eqref{eq:7a}.
One can see that $\theta \circ  f \circ \hat{\theta} \circ i_{B^{\psi}}=i_{\psi} \circ \psi \circ \psi^{-1}=i_{\psi}$. By the universal property of Hopf envelopes in \eqref{eq:hopfenvelope}, we have $\theta \circ  f \circ \hat{\theta}=f$.
\end{proof}

\begin{Cor}
\label{f-well-def}
The map $f$  induces a Hopf algebra map $\mc{H}(B^\psi)^{\hat \theta}\to \mc{H}(B)$. 
\end{Cor}

\begin{proof}
We use $\cdot$ for the multiplications in both $\mc{H}(B^\psi)$ and $\mc{H}(B)$, and $*$ for the multiplications in $\mc{H}(B^\psi)^{\hat \theta}$ and $\mc{H}(B)^\theta$. Using \Cref{fs-and-thetas}, for homogeneous elements $x, y\in \mathcal H(B^{\psi})^{\hat{\theta}}$, we have
\[
f(x * y) = f(x \cdot \hat\theta^{|x|}(y))= f(x) * f \hat \theta^{|x|}(y)= f(x) \cdot \theta^{|f(x)|} f \hat \theta^{|x|} (y)= f(x) \cdot \theta^{|x|} f \hat \theta^{|x|} (y)=f(x) \cdot f(y). \qedhere
\]
\end{proof}

The following result can also be proved directly by using the explicit construction of the Hopf envelope. 

\begin{thm}
\label{f-iso}
Let $B$ be a bialgebra satisfying the twisting conditions (\textbf{T1})-(\textbf{T2}). For any graded bialgebra automorphism $\psi$ of $B$, we have the following Hopf algebra isomorphism: 
\[\mc{H}(B^\psi) \cong  \mc{H}(B)^{\mathcal H(\psi)}.\]
\end{thm}

\begin{proof}
We show that the map $f$ constructed above is an isomorphism of Hopf algebras $\mc{H}(B^\psi) \xrightarrow{\cong} \mc{H}(B)^\theta$ by producing an inverse to $f$. Here again, by the universal property of $\mathcal H(B)$, there is a unique Hopf algebra map $g$ making the diagram below commute:
\[
\xymatrix{ 
B \ar[r]^-{i_B}\ar[dr]_-i & \mc{H} (B) \ar[d]^-g\\
& \mc{H} \left( B^\psi\right)^{\hat\theta}.
}
\]
By the fact that $(B^\psi)^{\psi^{-1}} \cong B$ and by \Cref{i-map-both}, $i$ is a map of bialgebras $B \cong (B^{\psi})^{\psi^{-1}} \to \mc{H}(B^\psi)^{\hat \theta}$ induced by $i_{B^\psi}: B^\psi \to \mc{H}(B^{\psi}).$ Now using \Cref{f-well-def}, we obtain the commutative diagram:
\[
\xymatrix{
B \ar[rr]^-{i_B} \ar[d]_-{i_B} \ar[rd]^-i && \mc{H}(B) \\
\mc{H}(B) \ar[r]_-g & \mc{H} \left( B^\psi\right)^{\hat\theta} \ar[ru]_-f. }
\]
The universal property of $\mathcal H(B)$ forces $f \circ g=\id$. The argument for $g \circ f = \id $ is symmetric. By the same argument as \Cref{f-well-def}, $g$ is a Hopf algebra map $\mc{H}(B)^\theta \to \mc{H}(B^\psi)$. 
\end{proof}

%%%%%%%%%%%%%%%%%%%%%%%%%%%%%%%%%%%%%%%%%
\subsection{Zhang twists vs.~2-cocycle twists of Hopf algebras}

For a Hopf algebra $H$, there is another way of deforming the multiplicative structure. This twist, called a 2-cocycle twist, is defined by a 2-cocycle $\sigma \in {\rm Hom}_\kk(H\otimes H,\kk)$, see \cite{Doi93, DT94, M2005}.  
We now relate the above discussion of Zhang twists to the 2-cocycle twists of Hopf algebras to provide a partial answer to \Cref{ques:twist}. 

We first recall some background on 2-cocycle twists. Let $H$ be a Hopf algebra. A \emph{left $H$-Galois object} is a left $H$-comodule algebra $A\neq 0$ such that if $\alpha: A\to H\otimes A$ denotes the coaction of $H$ on $A$, the linear map defined by the following composition 
\[
\xymatrix{ 
A\otimes A \ar[r]^-{\alpha\otimes \id} & H\otimes A\otimes A \ar[r]^-{\id \otimes m} & H\otimes A,}
\]
is an isomorphism of vector spaces, where $m : A \otimes A \rightarrow A$ is the multiplication map of $A$. A \emph{right $H$-Galois object} can be defined in an analogous way. If $K$ is another Hopf algebra, then an \emph{$(H,K)$-biGalois object} is an $H$-$K$-bicomodule algebra which is both a left $H$-Galois object and a right $K$-Galois object. A classical result of Schauenburg \cite{Sch1996} states that two Hopf algebras are Morita-Takeuchi equivalent (or their comodule categories are monoidally equivalent) if and only if there exists a biGalois object between them.

There is an important class of Hopf-Galois objects, called \emph{cleft} \cite{Bichon2014}. A \emph{right $H$-cleft object} is a right $H$-comodule algebra $A$ which admits an $H$-comodule isomorphism $\phi: H \xrightarrow{\sim} A$ that is also invertible with respect to the convolution product. Such a map $\phi$ can be chosen so that it preserves the unit; in this case, $\phi$ is called a \emph{section}. A right $H$-cleft object is a right $H$-Galois object that admits an $H$-comodule isomorphism $\phi: H \xrightarrow{\sim} A$. Given another Hopf algebra $K$, a \emph{left $K$-cleft object} and an \emph{$(K,H)$-bicleft object} are defined analogously. 

Cleft objects can be constructed from 2-cocycles as follows: A \emph{2-cocycle} on a Hopf algebra $H$ is a convolution invertible linear map $\sigma: H\otimes H\to \kk$ satisfying 
\begin{equation}\label{eq:2cocycle}
\begin{aligned}
   \sum \sigma(x_1, y_1)\sigma (x_2y_2, z)&=\sum \sigma(y_1, z_1)\sigma(x, y_2z_2)\quad \text{and}\\
   \sigma(x,1)&=\sigma(1,x)=\varepsilon(x),
\end{aligned}
\end{equation} 
for all $x,y,z\in H$. The set of all $2$-cocycles on $H$ is denoted by $Z^2(H)$. The convolution inverse of $\sigma$, denoted by $\sigma^{-1}$, satisfies
\begin{equation}\label{eq:2cocycleinverse}
\begin{aligned}
    \sum \sigma^{-1}(x_1y_1, z)\sigma^{-1}(x_2, y_2)&=\sum \sigma^{-1}(x, y_1z_1)\sigma^{-1}(y_2, z_2)\quad \text{and}\\
    \sigma^{-1}(x,1)&=\sigma^{-1}(1,x)=\varepsilon(x),
\end{aligned}
\end{equation}
for all $x,y,z\in H$. Given a pair $(A,\phi)$, consisting of a right $H$-cleft object $A$ together with a section $\phi: H \overset{\sim}{\longrightarrow} A$, consider the linear map $\sigma$ defined on $H\otimes H$ by 
\begin{equation}\label{eq:section2cocycle}
\sigma(x,y):=\sum \phi(x_1)\phi(y_1)\overline{\phi}(x_2y_2), 
\end{equation}
for any $x,y\in H$, where $\overline{\phi}$ is the convolution inverse of $\phi$. Then $\sigma$ takes values in $\kk=A^{{\rm co}H}$, and is a 2-cocycle on $H$ \cite[Theorem 11]{DT86}. Moreover, let $\!_\sigma H$ denote the right $H$-comodule algebra $H$ endowed with the original unit and deformed product
\[
x\cdot_\sigma y=\sum \sigma(x_1,y_1)x_2y_2,
\]
for any $x,y\in \!_\sigma H$. Then we have an isomorphism $\!_\sigma H\overset{\sim}{\longrightarrow} A$ via $y\mapsto \phi(y)$ as $H$-comodule algebras. Indeed every right $H$-cleft object arises in this way \cite[Proposition 1.4]{Mas94}. Similarly, we can define the left $H$-comodule algebra $H_{\sigma^{-1}}$ by using the convolution inverse of any 2-cocycle $\sigma$ on $H$. Again, every left $H$-cleft object  arises in this way.

Given a 2-cocycle $\sigma: H\otimes H\to\kk$, let $H^\sigma$ denote the coalgebra $H$ endowed with the original unit and deformed product
\[
x*_\sigma y:=\sum \sigma(x_1,y_1)\,x_2y_2\,\sigma^{-1}(x_3,y_3).
\]
By \cite{Doi93}, $H^\sigma$ is a bialgebra and is indeed a Hopf algebra with the deformed antipode $S^\sigma$ given in \cite[Theorem 1.6]{Doi93}. We call $H^\sigma$ the \emph{2-cocycle twist} of $H$ by $\sigma$. It is well-known that two Hopf algebras are 2-cocycle twists of each other if and only if there exists a bicleft object between them (e.g., see \cite{Sch1996}). 

\begin{proposition}
\label{twist-and-cleft}
Let $H$ be a Hopf algebra satisfying the twisting conditions (\textbf{T1})-(\textbf{T2}). Let $H^\phi$ be a right Zhang twist  of $H$ by a graded algebra automorphism $\phi$. 
\begin{enumerate}
  \item Suppose $\Delta\circ \phi=(\phi\otimes \id)\circ \Delta$. Then $H^\phi\cong \!_\sigma H$ is right $H$-cleft with a 2-cocycle $\sigma: H\otimes H\to \kk$ given by $\sigma(x,y)=\varepsilon(x)\varepsilon(\phi^{|x|}(y))$, for homogeneous elements $x,y\in H$.
    \item Suppose $\Delta\circ \phi=(\id\otimes \phi)\circ \Delta$. Then $H^\phi\cong H_{\sigma^{-1}}$ is left $H$-cleft with a 2-cocycle convolution inverse $\sigma^{-1}: H\otimes H\to \kk$ given by $\sigma^{-1}(x,y)=\varepsilon(x)\varepsilon(\phi^{|x|}(y))$, for homogeneous elements $x,y\in H$. 
\end{enumerate}
\end{proposition}
\begin{proof}
We prove part (1), and part (2) follows in a similar way. 
We first show that $H^\phi$ is a right Galois-object. Note that $H^\phi \cong H$, as graded vector spaces. As a result, the comultiplication $\Delta: H\to H\otimes H$ gives $H^\phi$ a right $H$-comodule structure via $\Delta^\phi: H^\phi\to H^\phi\otimes H$. If $\Delta\circ \phi=(\phi\otimes \id)\circ \Delta$, then
\[\small
    \Delta(x*y)=\Delta(x\phi^{|x|}(y))=\Delta(x)\Delta(\phi^{|x|}(y))=\Delta(x)( \phi^{|x|}\otimes \id)\Delta(y)=\sum x_1\phi^{|x_1|}(y_1)\otimes x_2y_2=\Delta^\phi(x)\Delta^\phi(y),
\]
for any homogeneous elements $x,y\in H^{\phi}$, so $H^\phi$ is a right $H$-comodule algebra. Now, it can be checked that the map used in the definition of a right $H$-Galois object  
\[
H^\phi\otimes H^\phi\xrightarrow{\id\otimes \Delta^\phi} H^\phi\otimes H^\phi\otimes H\xrightarrow{m\otimes \id} H^\phi\otimes H,
\]
given by $x\otimes y\mapsto \sum x\phi^{|x|}(y_1)\otimes y_2$, is bijective with inverse $x\otimes y\mapsto \sum x\phi^{|x|-|y|}(S(y_1))\otimes y_2$, where $S$ is the antipode of $H$. Thus $H^\phi$ is a right $H$-Galois object. Moreover, the identity map $H=H^\phi$ is an isomorphism of right $H$-modules. Therefore, we conclude that $H^\phi$ is cleft \cite[Theorem 9]{DT86}.

In fact, the $H$-comodule isomorphism $\phi: H\to H^\phi$ is indeed invertible with respect to the convolution product in ${\rm Hom}_\kk(H,H^\phi)$ with inverse $\overline{\phi}$ given by
\[
\overline{\phi}(x)=\phi^{1-|x|}(S(x)),
\]
for any homogeneous element $x\in H$. Then by \eqref{eq:section2cocycle}, the 2-cocycle $\sigma: H\otimes H\to \kk$ associated with the $H$-cleft object $H^\phi$ is given by
\begin{align*}
\sigma(x,y)&= \sum \phi(x_1)*\phi(y_1)*\overline{\phi}(x_2y_2)\\
&= \sum (\phi(x_1)\phi^{1+|x|}(y_1))* \phi^{1-|x_2y_2|}(S(x_2y_2))\\
&= \sum \phi(x_1)\phi^{1+|x|}(y_1)\phi(S(x_2y_2))\\
&= \sum \phi(x_1\phi^{|x|}(y_1)S(y_2)S(x_2))\\
\overset{{(\bf{P1})}}&{=} \sum \phi(x_1\sum(\phi^{|x|}(y)_1S(\phi^{|x|}(y)_2))S(x_2))\\
&=\phi\left(\sum x_1S(x_2)\right)\varepsilon(\phi^{|x|}(y))\\
&= \varepsilon(x)\varepsilon(\phi^{|x|}(y)). \qedhere
\end{align*}
\end{proof}

We now relate the Zhang twist of a graded Hopf algebra $H$ by a twisting pair to a 2-cocycle twist of $H$.

\begin{proposition}
\label{twist-and-2-cocycle}
Let $H$ be a Hopf algebra satisfying the twisting conditions (\textbf{T1})-(\textbf{T2}). For any twisting pair $(\phi_1,\phi_2)$ of $H$, we have the following.
\begin{enumerate}
    \item The map $\phi_1 \circ \phi_2$ is a graded Hopf automorphism of $H$.
    \item The linear map $\sigma: H \otimes H \rightarrow \kk$ defined by 
    \[\sigma(x,y)=\varepsilon(x)\varepsilon(\phi_2^{|x|}(y)), \text{ for any homogeneous elements }x,y \in H,\] 
    is a $2$-cocycle, whose convolution inverse $\sigma^{-1}$ is given by 
    \[\sigma^{-1}(x,y)=\varepsilon(x)\varepsilon(\phi_1^{|x|}(y)), \text{ for any homogeneous elements }x,y \in H. \]
    \item The 2-cocycle twist $H^\sigma \cong H^{\phi_1 \circ \phi_2}$ is a right Zhang twist, where the $2$-cocyle $\sigma$ is described in (2).
\end{enumerate} 
As a consequence, $H$ and $H^{\phi_1 \circ \phi_2}$ are  Morita-Takeuchi equivalent with bi-cleft object given by $H^{\phi_1}$. 
\end{proposition}

\begin{proof}
(1): By our assumptions, we know that $\phi_1 \circ \phi_2$ is a graded algebra automorphism of $H$ and commutes with the counit $\varepsilon$. So, it suffices to show that $\phi_1 \circ \phi_2$ is compatible with the coassociativity axiom. This follows from our twisting pair conditions:
\begin{align*}
    \Delta \circ \phi_1 \circ \phi_2 &\overset{(\textbf{P1})}{=}(\id \otimes \phi_1) \circ \Delta \circ \phi_2 \overset{(\textbf{P1})}{=} (\id \otimes \phi_1) \circ (\phi_2\otimes \id) \circ \Delta \\ 
    & \overset{(\textbf{P4})}{=} (\phi_2\otimes \phi_1) \circ (\phi_1\otimes \phi_2) \circ \Delta \overset{(\textbf{P3})}{=} ((\phi_1 \circ \phi_2)\otimes (\phi_1 \circ \phi_2)) \circ \Delta.
\end{align*}
It follows that $\phi_1 \circ \phi_2$ is a Hopf automorphism. 

(2): We show that the maps $\sigma$ and $\sigma^{-1}$ defined in the statement are inverses to each other with respect to the convolution product $*$ in ${\rm Hom}_\kk(H\otimes H,\kk)$. Indeed, for any homogeneous elements $x,y\in H$, we have 
\begin{align*}
    (\sigma*\sigma^{-1})(x,y)&=\sum \sigma(x_1,y_1)\sigma^{-1}(x_2,y_2)\\
    &=\sum \varepsilon(x_1)\varepsilon(\phi_2^{|x|}(y_1))\varepsilon(x_2)\varepsilon(\phi_1^{|x|}(y_2))\\
    &=\sum \varepsilon(x_1\varepsilon(x_2)) \varepsilon(\phi_2^{|x|}(y_1)\phi_1^{|x|}(y_2))\\
    \overset{(\textbf{P1})}&= \sum \varepsilon(x) \varepsilon((\phi_1\phi_2)^{|x|}(y)_1(\phi_1\phi_2)^{|x|}(y)_2)\\
    &= \sum \varepsilon(x) \varepsilon((\phi_1\phi_2)^{|x|}(y))\\
    \overset{(\textbf{P2})}&{=}\varepsilon(x)\varepsilon(y).
\end{align*}
Similarly, we have $(\sigma^{-1}*\sigma)(x,y)=\varepsilon(x)\varepsilon(y)$.
Secondly, it is clear that $\sigma(1,x)=\sigma(x,1)=\varepsilon(x)$. For any homogeneous elements $x,y,z \in H$, we have 
\begin{align*}
\sum \sigma(x_1,y_1)\sigma(x_2y_2,z)&=\sum \varepsilon(x_1)\varepsilon(\phi_2^{|x|}(y_1))\varepsilon(x_2y_2)\varepsilon(\phi_2^{|xy|}(z))\\
&=\sum \varepsilon(x_1\varepsilon(x_2))\varepsilon(\phi_2^{|x|}(y_1\varepsilon(y_2)))\varepsilon(\phi_2^{|xy|}(z))\\
&= \sum \varepsilon(x)\varepsilon(\phi_2^{|x|}(y))\varepsilon(\phi_2^{|xy|}(z)),
\end{align*}
and 
\begin{align*}
    \sum \sigma(y_1,z_1)\sigma(x,y_2z_2)&=\sum \varepsilon(y_1)\varepsilon(\phi_2^{|y|}(z_1))\varepsilon(x)\varepsilon(\phi_2^{|x|}(y_2z_2))\\
    &=\sum \varepsilon(x)\varepsilon(\phi_2^{|x|}(\varepsilon(y_1)y_2))\varepsilon(\phi_2^{|y|}(z_1)) \varepsilon(\phi_2^{|x|}(z_2))\\
    &=\sum \varepsilon(x)\varepsilon(\phi_2^{|x|}(y))\varepsilon(\phi_2^{|y|}(z_1)\phi_2^{|x|}(z_2))\\
\overset{(\textbf{P1})}&=\sum \varepsilon(x)\varepsilon(\phi_2^{|x|}(y))\varepsilon(\phi_2^{|y|}(z)_1\phi_2^{|x|}(\phi_2^{|y|}(z)_2))\\
&=\sum \varepsilon(x)\varepsilon(\phi_2^{|x|}(y))\varepsilon(\phi_2^{|x|}(\varepsilon(\phi_2^{|y|}(z)_1)\phi_2^{|y|}(z)_2)) \\
&=\sum \varepsilon(x)\varepsilon(\phi_2^{|x|}(y))\varepsilon(\phi_2^{|xy|}(z)).
\end{align*}
Hence, $\sum \sigma(x_1,y_1)\sigma(x_2y_2,z)=\sum \sigma(y_1,z_1)\sigma(x,y_2z_2)$ for all homogeneous elements $x,y,z\in H$. We conclude that $\sigma$ is a 2-cocycle on $H$ with the given convolution inverse $\sigma^{-1}$. 

(3): Now it is straightforward to check that the identity map on $H$ induces an isomorphism of Hopf algebras between the Zhang twist $H^{\phi_1\circ \phi_2}$ and the 2-cocycle twist $H^\sigma$. Indeed by definition $\id: H^{\phi_1\circ \phi_2}\xrightarrow{\sim} H^\sigma$ is a coalgebra map. So it remains to show it is also an algebra map, which follows from the fact that
\begin{align*}
 x*_\sigma y&= \sum \sigma(x_1,y_1)x_2y_2\sigma^{-1}(x_3,y_3)\\
 &=\sum \varepsilon(x_1)\varepsilon(\phi_2^{|x|}(y_1))x_2y_2\varepsilon(x_3)\varepsilon(\phi_1^{|x|}(y_3))\\
 &= \sum x\varepsilon(\phi_2^{|x|}(y_1))y_2\varepsilon(\phi_1^{|x|}(y_3))\\
\overset{(\textbf{P1})}&= \sum x\varepsilon(\phi_2^{|x|}(y_1)_1)\phi_2^{|x|}(y_1)_2\varepsilon(\phi_1^{|x|}(y_2))\\ 
&= \sum x\phi_2^{|x|}(y_1)\varepsilon(\phi_1^{|x|}(y_2))\\
\overset{(\textbf{P1})}&= \sum x\phi_1^{|x|}\phi_2^{|x|}(y)_1\varepsilon(\phi_1^{|x|}\phi_2^{|x|}(y)_2)\\
&= x\phi_1^{|x|}\phi_2^{|x|}(y)\\
&= x*_{(\phi_1\circ \phi_2)} y,
\end{align*}
for all homogeneous elements $x,y\in H$. Hence, $H$ and $H^{\phi_1 \circ \phi_2}$ are Morita-Takeuchi equivalent. Finally, by \Cref{twist-and-cleft}, $H^{\phi_1}$ is a left $H$-cleft object. By a left Zhang twist version of \Cref{twist-and-cleft}, we can argue similarly that $H^{\phi_1}$ is a right $\!^{\phi_2^{-1}}(H^{\phi_1})=H^{\phi_1 \circ \phi_2}$-cleft object. So $H^{\phi_1}$ is a $(H,H^{\phi_1 \circ \phi_2})$-bicleft object.
\end{proof}

\begin{thm}
\label{thm:Zhang and 2-cocycle twists}
Let $B$ be a bialgebra satisfying the twisting conditions (\textbf{T1})-(\textbf{T2}). For any twisting pair $(\phi_1,\phi_2)$ of $B$, there is a unique twisting pair $(\mathcal H(\phi_1),\mathcal H(\phi_2))$ of the Hopf envelope $\mathcal H(B)$ extending $(\phi_1,\phi_2)$. Moreover, the 2-cocycle twist $\mathcal H(B)^\sigma$, with the 2-cocycle $\sigma: \mathcal H(B)\otimes \mathcal H(B)\to \kk$ given by \[\sigma(x,y)=\varepsilon(x)\varepsilon(\mathcal H(\phi_2)^{|x|}(y)), \quad 
\text{for any homogeneous elements } x,y\in \mathcal H(B),\]
is the right Zhang twist $\mathcal H(B)^{\mathcal H(\phi_1\circ \phi_2)}$.
\end{thm}

\begin{proof}
This follows from \Cref{cor:lifting} and \Cref{twist-and-2-cocycle}. 
\end{proof}

\begin{remark}
It is clear that not every 2-cocycle twist of a Hopf algebra can be realized as a Zhang twist. For instance, any Hopf algebra, when viewed as a $\mathbb Z$-graded algebra concentrated in degree zero, satisfies the twisting conditions (\textbf{T1})-(\textbf{T2}). But in this case, all Zhang twists are trivial. 
\end{remark}

%%%%%%%%%%%%%%%%%%%%%%%%%%%%%%%%%%%%%%%%%
\section{Manin's universal quantum groups}
\label{sec:Manin}

In this section, we examine Zhang twists of Manin's universal quantum groups of quadratic algebras, where automorphisms of such universal quantum groups come from those of the underlying quadratic algebras, and connect them to the 2-cocycle twists of these universal quantum groups.

%%%%%%%%%%%%%%%%%%%%%%%%%%%%%%%%%%%%%%%%%
\subsection{Preliminaries on Manin's universal bialgebras and universal quantum groups}

We first recall Manin's construction of the universal bialgebra and the universal quantum group as described in \cite{Manin2018}. Suppose $A$ is a quadratic algebra. We use the presentation $A=\kk\langle A_1\rangle /(R(A))$, where $R(A)\subseteq A_1\otimes A_1$ with $\dim_\kk A_1<\infty$. We denote by $\underline{\rm end}^l(A)$ the \emph{universal bialgebra} that left coacts on $A$ via $\rho_A: A\to  \underline{\rm end}^l(A)\otimes A$ preserving the $\mathbb Z$-grading of $A$, i.e., $\rho_A: A_n\to \underline{\rm end}^l(A)\otimes A_n$, and satisfies following
universal property: If $B$ is any bialgebra that left
coacts on $A$ via $\tau: A\to B\otimes A$ such that the coaction of $B$ preserves the grading of $A$, then there is a unique bialgebra map $f: \underline{\rm end}^l(A)\to B$ such that the following diagram commutes:
\[
\xymatrix{
A\ar[r]^-{\rho}\ar[dr]_-{\tau} &  \underline{\rm end}^l(A)\otimes A\ar[d]^-{f\otimes \id} \\
& B\otimes A.
}
\]
Similarly, we can define $\underline{\rm end}^r(A)$ as the universal bialgebra that right coacts on $A$ universally preserving the grading of $A$. For Hopf algebra coactions, we define $\underline{\rm aut}^l(A)$ 
(resp.~$\underline{\rm aut}^r(A)$) to be the \emph{universal quantum group} that coacts on $A$ universally on the left (resp.~on the right) preserving the grading of $A$. In what follows, whenever we write $\underline{\rm end}$ or $\underline{\rm aut}$ without the superscript, the statements hold for both left and right versions.

For two quadratic algebras $A=\kk\langle A_1\rangle/(R(A))$ and $B=\kk\langle B_1\rangle/(R(B))$, the \emph{bullet product} of $A$ and $B$ is defined as 
\begin{equation}\label{eqn:bullet} 
A\bullet B:=\frac{\kk\langle A_1\otimes B_1\rangle}{\left(S_{(23)}\left(R(A)\otimes R(B)\right)\right)},\end{equation} 
where $S_{(23)}: A_1\otimes A_1\otimes B_1\otimes B_1\to A_1\otimes B_1\otimes A_1\otimes B_1$ is the flip of the middle two tensor factors in the 4-fold tensor product \cite[Section 4.2]{Manin2018}. 

\begin{lemma}
\label{bullet-twist} 
Let $A=\kk\langle A_1\rangle/(R(A))$, $B=\kk\langle B_1\rangle/(R(B))$ be quadratic algebras, and $\phi: A \to A$, $\psi: B \to B$ be graded automorphisms. Then, we have an algebra isomorphism  
\[A^\phi \bullet B^\psi \cong (A \bullet B)^{\phi \bullet \psi},\]
where $\phi \bullet \psi$ is the graded automorphism of $A \bullet B$ generated by the degree-1 component $\phi_1 \otimes \psi_1: A_1 \otimes B_1 \to A_1 \otimes B_1$. 
\end{lemma}
\begin{proof}
We note that
\begin{align*}
(\phi_1 \otimes \psi_1 \otimes \phi_1 \otimes \psi_1) \circ  S_{(23)} (R(A) \otimes R(B)) &= S_{(23)} ( \phi( R(A) ) \otimes \psi(R(B)))\\
&= S_{(23)} ( R(A) \otimes R(B)).
\end{align*}
Therefore, $\phi \bullet \psi$ preserves the degree-2 relations in $A \bullet B$, and is a graded automorphism of $A \bullet B$. 
For the isomorphism $A^\phi \bullet B^\psi \cong (A \bullet B)^{ \phi \bullet \psi}$, by \cite[Lemma 5.1]{Mori-Smith2016} the twist $(A \bullet B)^{\phi \bullet \psi}$ has relations 
\begin{align*} 
\left( \left( \id \otimes \id\right) \otimes \left( \phi_1 \otimes \psi_1\right)^{-1} \right) \circ S_{(23)} \left( R(A) \otimes R(B) \right) &= S_{(23)} \left( \left(\id \otimes \phi_1^{-1} \otimes \id \otimes \psi_1^{-1} \right) \left( R(A) \otimes R(B) \right) \right) \\
&= S_{(23)} \left( \left( \id \otimes \phi_1^{-1} \right) \left( R(A) \right) \otimes \left( \id \otimes \psi_1^{-1} \right) \left( R(B) \right) \right).  
\end{align*}
The last line gives exactly the relations of $A^\phi \bullet B^\psi$.
\end{proof}

In particular, we may consider the bullet product $A^!\bullet A$ of a quadratic algebra $A$ and its Koszul dual $A^!$  \cite{Shelton-Tingey}, which is a bialgebra with matrix comultiplication defined below on the generators $(A^!)_1\otimes A_1$. Choose a basis $\{x_1,\ldots,x_n\}$ for $A_1$ and let $\{x^1,\ldots,x^n\}$ be the dual basis for $(A^!)_1=(A_1)^*$. Write $z^{k}_j= x^k \otimes x_j\in (A^!)_1\otimes A_1$ as the generators for $A^!\bullet A$. Then the comultiplication on $A^!\bullet A$ is given by 
\begin{equation}
\label{eq:comultManin}
    \Delta(z^{k}_j)=\sum_{1 \leq i \leq n} z^i_j\otimes z^k_i.
\end{equation}

\noindent We now list some of the basic properties of universal bialgebras and universal quantum groups. In what follows, for an algebra (resp.~a coalgebra) $A$, we denote by $A^{op}$ a $\kk$-vector space $A$ with opposite multiplication, and by $A^{cop}$ a $\kk$-algebra $A$ with opposite comultiplication, respectively.

\begin{proposition}\cite{Manin2018}
\label{prop:2}
Given a quadratic algebra $A$, we have isomorphisms of universal bialgebras:
\begin{enumerate}
\item $\underline{\rm end}(A^{op})\cong \underline{\rm end}(A)^{op}$ and $\underline{\rm end}(A^!)\cong \underline{\rm end}(A)^{cop}$. 
\item $\underline{\rm end}^r(A)\cong A\bullet A^!$ and $\underline{\rm end}^l(A)\cong A^!\bullet A$. 
\item $\underline{\rm end}^r(A)\cong \underline{\rm end}^l(A)^{cop}$. 
\item $\underline{\rm end}^r(A)\cong \underline{\rm end}^l(A^!)$ and $\underline{\rm end}^l(A)\cong \underline{\rm end}^r(A^!)$.
\end{enumerate}
\end{proposition}

The following results can be derived from \Cref{prop:2} and the universal property of a Hopf envelope. 

\begin{proposition}\label{prop:30}
Given a quadratic algebra $A$, we have isomorphisms of universal quantum groups:
\begin{enumerate}
\item[(1)] $\underline{\rm aut}(A)\cong \mathcal H(\underline{\rm end}(A))$.
\item[(2)] $\underline{\rm aut}(A^{op})\cong \underline{\rm aut}(A)^{op}$ and $\underline{\rm aut}(A^!)\cong \underline{\rm aut}(A)^{cop}$.
\item[(3)] $\underline{\rm aut}^r(A)\cong \underline{\rm aut }^l(A)^{cop}$.
    \item[(4)] $\underline{\rm aut}^r(A)\cong \underline{\rm aut}^l(A^!)$ and $\underline{\rm aut}^l(A)\cong \underline{\rm aut}^r(A^!)$.
\end{enumerate}
\label{pro:endandaut}
\end{proposition}

As a consequence, we have the following lemma. 

\begin{lemma}
\label{Manin-twisting-conditions}
For any quadratic algebra $A$, its universal bialgebra $\underline{\rm end}(A)$ and its universal quantum group $\underline{\rm aut}(A)$ both satisfy the twisting conditions (\textbf{T1})-(\textbf{T2}).
\end{lemma}

\begin{proof}
By \Cref{prop:2} (2) and the comultiplication \eqref{eq:comultManin}, $\underline{\rm end}(A)$ satisfies (\textbf{T1}) and (\textbf{T2}). Moreover, by \Cref{prop:30} (1), $\underline{\rm aut}(A)$ is the Hopf envelope of $\underline{\rm end}(A)$. The result follows from \Cref{t1-t3-hopf-env}. 
\end{proof}

%%%%%%%%%%%%%%%%%%%%%%%%%%%%%%%%%%%%
\subsection{Zhang twists and 2-cocycle twists of Manin's universal quantum groups}

Let $A$ be a quadratic algebra. We denote by ${\rm Aut}_{\gr}(A)$ the group of all graded automorphisms of $A$. Let $\underline{\rm end}^l(A)$ be the universal bialgebra that left coacts on $A$ universally via $\rho_A: A\to \underline{\rm end}^l(A)\otimes A$, preserving the grading. Since $\underline{\rm end}^l(A)$ is again a quadratic algebra, we can write ${\rm Aut}_{\gr}(\underline{\rm end}^l(A))$ for the group of all its graded automorphisms. We will use the following universal property of 
$\underline{\rm end}^l(A)$ \emph{as a universal graded algebra}, whose proof is based on \cite[Lemma 6.6]{Manin2018}. 

\begin{lemma}
\label{lem:30}
Let $A$ and $B$ be quadratic algebras and $\rho: A\to B\otimes A$ be an algebra morphism with $\rho(A_1)\subseteq B_1\otimes A_1$. Then there exists a unique graded algebra morphism $\gamma: \underline{\rm end}^l(A)\to B$ such that the following diagram commutes: 
\[
\xymatrix{
A\ar[r]^-{\rho_A}\ar[dr]_-{\rho} & \underline{\rm end}^l(A)\otimes A\ar[d]^-{\gamma\otimes \id} \\
& B\otimes A.
}
\]
\end{lemma}
For any $\phi\in {\rm Aut}_{\gr}(A)$, the following composition of algebra morphisms
\[\rho^\phi_A:=\left( A\xrightarrow{\phi} A\xrightarrow{\rho_A} \underline{\rm end}^l(A)\otimes A\right)\]
satisfies $\rho^\phi_A(A_1)\subseteq \underline{\rm end}^l(A)_1\otimes A_1$. Therefore by \Cref{lem:30}, there is a unique algebra endomorphism $\varphi$ of $\underline{\rm end}^l(A)$ such that the following diagram commutes:
\begin{equation}\label{eq:dia}
\xymatrix{
A\ar[r]^-{\rho_A}\ar[d]_-{\phi} & \underline{\rm end}^l(A)\otimes A \ar[d]^-{\varphi\otimes \id}\\
A\ar[r]^-{\rho_A}&  \underline{\rm end}^l(A)\otimes A.
}
\end{equation}
In what follows, we will denote such a map by $\underline{\rm end}^l(\phi)= \varphi$.

On the other hand, note that any graded automorphism $\phi$ of $A$ naturally gives a graded automorphism of the Koszul dual $A^!$ by extending the dual map $(\phi|_{A_1})^*$ on the generators of $(A^!)_1$ to $A^!$. We denote such a dual graded automorphism of $A^!$ by $\phi^!$. One can easily check that $(-)^!: {\rm Aut}_{\gr}(A)\to {\rm Aut}_{\gr}(A^!)$ is a group anti-isomorphism.

By \Cref{prop:2}, we have isomorphisms $
\underline{\rm end}^l(A) \cong A^!\bullet A \cong \underline{\rm end}^r(A^!)$. 
A similar result to \Cref{lem:30} regarding the universal right coaction of $\underline{\rm end}^r(A^!)$ on $A^!$ yields another graded algebra endomorphism of $\underline{\rm end}^r(A^!) \cong \underline{\rm end}^l(A)$, denoted by $\underline{\rm end}^r(\phi^!)$, such that the following diagram commutes:
\[
\xymatrix{
A^!\ar[r]^-{\rho_{A^!}}\ar[d]_-{\phi^!} & A^!\otimes \underline{\rm end}^r(A^!)\ar[d]^-{\id \otimes \underline{\rm end}^r(\phi^!)}\\
A^!\ar[r]^-{\rho_{A^!}}&  A^!\otimes \underline{\rm end}^r(A^!).
}
\]
Moreover, by the universal property, one can show that $\underline{\rm end}^l(-)$ and $\underline{\rm end}^r((-)^{-1})^!$ are indeed injective group homomorphisms from ${\rm Aut}_{\gr}(A)$ to ${\rm Aut}_{\gr}(\underline{\rm end}^l(A))$ with commuting images. Similarly, we can define $\underline{\rm end}^r(\phi)$ and $\underline{\rm end}^l(\phi^!)$ as graded algebra endomorphisms of 
$\underline{\rm end}^r(A) \cong A\bullet A^! \cong \underline{\rm end}^l(A^!)$.  

We first investigate the Koszul dual of a Zhang twist of a quadratic algebra. In fact, this is a twist by the dual automorphism on the other side. This phenomenon extends to $N$-Koszul algebras by an analogous proof.

\begin{lemma}
\label{lemma:twist-of-dual}
\cite[Proposition 5.5]{UK2010}
If $A$ is a quadratic algebra and $\phi \in \Aut_{\gr}(A)$ is a graded automorphism of $A$, then the dual automorphism $\phi^!$ is a graded automorphism of $A^!$ so that $$\left( \prescript{\phi}{}{A}\right)^! \cong \prescript{ \left( \phi^{-1} \right)^{!}}{}{\left( A^!\right)} \cong \left( A^!\right)^{\phi^!}.$$\end{lemma}

The following result provides us with a complete classification of all twisting pairs on Manin's universal bialgebras. As a consequence of \Cref{cor:lifting} and \Cref{prop:30} (1), we obtain twisting pairs on the corresponding universal quantum groups. 

\begin{lemma}
\label{Manin-twisting-pairs}
Let $A$ be a quadratic algebra. For any $\phi\in {\rm Aut}_{\gr}(A)$, we have: 
\begin{enumerate}
    \item $(\underline{\rm end}^r((\phi^{-1})^!), \underline{\rm end}^l(\phi))$ is a twisting pair for $\underline{\rm end}^l(A)$. 
    \item $(\underline{\rm end}^r(\phi), \underline{\rm end}^l((\phi^{-1})^!))$ is a twisting pair for $\underline{\rm end}^r(A)$. 
\end{enumerate}
Moreover, every twisting pair of Manin's universal bialgebra and quantum group is given in this way.
\end{lemma}
\begin{proof}

We check that conditions (\textbf{P1}) and (\textbf{P2}) from \Cref{defn:twisting pair} hold for (1). A similar argument can be made for (2). By \Cref{prop:2}, we know that $\underline{\rm end}^l(A)=A^!\bullet A$, so that $\underline{\rm end}^l(A)_1=A^*_1\otimes A_1$. By \Cref{lem:30}, one can check
\begin{align}\label{eq:6}
 \underline{\rm end}^l(\phi)|_{A^*_1\otimes A_1}=\id_{A^*_1}\otimes (\phi|_{A_1})   \qquad \text{and} \qquad \underline{\rm end}^r((\phi^{-1})^*)|_{A^*_1\otimes A_1}=(\phi^{-1}|_{A_1})^*\otimes\id_{A_1}.
\end{align}
It suffices to verify the conditions (\textbf{P1})-(\textbf{P2}) on the generating space $A_1^* \otimes A_1$ of $\underline{\rm end}^l(A)$, which we will do using \eqref{eq:6}. As before, we choose a basis $\{x_1,\ldots,x_n\}$ for $A_1$ with a dual basis $\{x^1,\ldots,x^n\}$ for $A_1^*$ and write 
$z^k_j=x^k \otimes x_j \in A_1^*\otimes A_1$
as the degree one generators of $\underline{\rm end}^l(A)$ for $1 \leq j,k \leq n$ (see \cite[Chapter 6]{Manin2018}). The comultiplication on $\underline{\rm end}^l(A)$ is given by \eqref{eq:comultManin}. Moreover, \eqref{eq:6} yields 
\begin{align}
 \underline{\rm end}^l(\phi)(z^k_j)&=\underline{\rm end}^l(\phi)(x^k\otimes x_j)=\sum_{1\leq s\leq n}x^k\otimes a_{js}x_s=\sum_{1\leq s\leq n} a_{js}z^k_s,\ \text{and} \label{eq:7}\\
 \underline{\rm end}^r((\phi^{-1})^*)(z^k_j)&=\underline{\rm end}^r((\phi^{-1})^*)(x^k\otimes x_j)=\sum_{1\leq l\leq n}b_{lk}x^l\otimes x_j=\sum_{1\leq l\leq n}b_{lk}z^l_j, \label{eq:e77}
\end{align}
where $\phi(x_j)=\sum_{1\leq s\leq n}a_{js}x_s$ and $(\phi^{-1})^*(x^j)=\sum_{1\leq l\leq n} b_{lj}x^l$. 
Here the matrices $(a_{ij})$ and $(b_{kl})$ are inverse to each other. Thus, we have for (\textbf{P1}):
\begin{align*}
 (\underline{\rm end}^l(\phi)\otimes \id)\circ \Delta (z^k_j)&= (\underline{\rm end}^l(\phi)\otimes \id)\left(\sum_{1\leq i\leq n} z^i_j\otimes z_i^k\right)=\sum_{1\leq i\leq n} \underline{\rm end}^l(\phi)(z^i_j)\otimes z_i^k=\sum_{1\leq i,l\leq n} a_{jl}z^i_l\otimes z^k_i\\
 &=\sum_{1\leq l\leq n}a_{jl}\Delta(z^k_l)=\Delta \left(\sum_{1\leq l\leq n}a_{jl}z^k_l\right)= \Delta\circ \underline{\rm end}^l(\phi)(z^k_j),
\end{align*}
and 
\begin{align*}
    (\id \otimes \ \underline{\rm end}^r(\phi^{-1})^!)\circ \Delta(z^k_j)&=\sum_{1\leq i\leq n} z^i_j\otimes \underline{\rm end}^r(\phi^{-1})^!(z^k_i)=\sum_{1\leq i,l\leq n} b_{lk} z^i_j\otimes z^l_i=\sum_{1\leq l\leq n} b_{lk}  \Delta(z^l_j)\\
    &=\Delta\left(\sum_{1\leq l\leq n} b_{lk}z^l_j\right)=\Delta\circ \underline{\rm end}^r(\phi^{-1})^!(z_j^k).
    \end{align*}
And for (\textbf{P2}):
\begin{align*}
    \varepsilon (\underline{\rm end}^r(\phi^{-1})^!\circ\underline{\rm end}^l(\phi)(z^k_j))&=\varepsilon \left( \underline{\rm end}^r(\phi^{-1})^! \left(\sum_{1\leq i\leq n}a_{ji}z^k_i\right) \right) =\varepsilon\left(\sum_{1\leq i,l\leq n} a_{ji}b_{lk}z_{i}^l \right) \\
    &=\sum_{1\leq i,l\leq n} a_{ji}b_{lk}\varepsilon(z_{i}^l)=\sum_{1\leq i,l\leq n} \delta_{il}a_{ji}b_{lk}
    =\sum_{1\leq i\leq n} a_{ji}b_{ik}=\delta_{jk}=\varepsilon(z^k_j).
 \end{align*}

Finally, suppose $(\phi_1,\phi_2)$ is any twisting pair of $\underline{\rm end}^l(A)=A^!\bullet A$. By (\textbf{P1}),  $\phi_2:A^!\bullet A\to A^!\bullet A$ is a right comodule map, where the comodule structure on $A^!\bullet A$ is given by its comultiplication $\Delta$. Write the degree one part of $A^!\bullet A$ as $C=A_1^*\otimes A_1$, which is isomorphic to the matrix coalgebra $M_n(\kk)^*$. By a direct computation, using (\textbf{P1}), and the definition of the comultiplication on $C$, we can write \[\phi_2|_C={\id}\otimes \phi \quad \text{for some $\phi\in \GL(A_1)$.}\] 
By \Cref{lem:winding}, we have
\[
\phi_1|_C=(\phi^{-1})^*\otimes {\id},
\]
where $(\phi^{-1})^*\in \GL(A_1^*)$. In view of \eqref{eq:6}, it remains to show that $\phi:A_1\to A_1$ can be extended to an algebra automorphism of $A$. We have
 \begin{align*}
     S_{(23)}\left(R(A)^\perp\otimes R(A)\right)&=R(A^!\bullet A)=(\phi_2\otimes \phi_2)(R(A^!\bullet A))=  (\id \otimes \phi \otimes \id \otimes \phi)\left(S_{(23)}(R(A^!)\otimes R(A))\right)\\
     &= S_{(23)}\left(R(A)^\perp\otimes (\phi\otimes \phi)R(A)\right).
 \end{align*}
Hence $(\phi\otimes \phi)R(A)=R(A)$ and $\phi$ can be extended to a graded algebra automorphism of $A=\kk\langle A_1\rangle/(R(A))$. We can similarly argue for $\underline{\rm end}^r(A)=A\bullet A^!$. Finally, twisting pairs of $\underline{\rm aut}(A)$ can be derived from \Cref{prop:30} and \Cref{cor:lifting}.
\end{proof}

As a consequence, the universal bialgebra of the Zhang twist of any quadratic algebra is isomorphic to some Zhang twist of the universal bialgebra of the original quadratic algebra.

\begin{thm}
\label{thm:twist-end}
Let $A$ be a quadratic algebra and $\phi$ be any graded automorphism of $A$. Then we have bialgebra isomorphisms
\[
\underline{\rm end}^l (A^\phi) \cong \underline{\rm end}^l(A)^{\underline{\rm end}^l(\phi)\circ \underline{\rm end}^r((\phi^{-1})^!)}\qquad \text{and} \qquad \underline{\rm end}^r (A^\phi) \cong \underline{\rm end}^r(A)^{\underline{\rm end}^r(\phi)\circ \underline{\rm end}^l((\phi^{-1})^!)}.
\]
\end{thm}

\begin{proof}
We show this for $\underline{\rm end}^l$, as the argument for $\underline{\rm end}^r$ follows similarly. By \Cref{Manin-twisting-conditions}, $\underline{\rm end}^l(A)$ satisfies the twisting conditions. Then \Cref{Manin-twisting-pairs} implies that $(\underline{\rm end}^r((\phi^{-1})^!),\underline{\rm end}^l(\phi))$ is a twisting pair for $\underline{\rm end}^l(A)$. By \Cref{twist-and-2-cocycle} (1) and (\textbf{P3}), we know that $\underline{\rm end}^r((\phi^{-1})^!) \circ \underline{\rm end}^l(\phi) =  \underline{\rm end}^l(\phi)\circ \underline{\rm end}^r((\phi^{-1})^!)$ is a graded bialgebra automorphism of $\underline{\rm end}^l(A)$. So, \Cref{zhang-twist-of-bialgebra} shows that $\underline{\rm end}^l(A)^{\underline{\rm end}^l(\phi)\circ \underline{\rm end}^r((\phi^{-1})^!)}$ is a well-defined bialgebra. Finally, by \Cref{prop:2} (2), \Cref{bullet-twist}, and a right twist version of \Cref{lemma:twist-of-dual}, we have algebra isomorphisms 
\begin{align*}
\uqba^l(A^\phi) &\cong (A^\phi)^! \bullet A^\phi\cong (A^!)^{(\phi^{-1})^!} \bullet A^\phi\overset{\eqref{eq:6}}{\cong}  (A^! \bullet A)^{ \underline{\rm end}^l(\phi)\circ \underline{\rm end}^r((\phi^{-1})^!)}\cong \uqba^l(A)^{\underline{\rm end}^l(\phi)\circ \underline{\rm end}^r((\phi^{-1})^!)},
\end{align*}
and it is straightforward to check that all isomorphisms respect the coalgebra structures as well.  
\end{proof}

Since $\mathcal H(\underline{\rm end}(A))\cong \underline{\rm aut}(A)$, for any graded bialgebra automorphism $\psi$ of $\underline{\rm end}(A)$, there is a unique graded Hopf automorphism, denoted by $\mathcal H(\psi)$, of $\underline{\rm aut}(A)$ such that the following diagram commutes:
\[
\xymatrix{
\underline{\rm end}(A)\ar[r]\ar[d]_-{\psi} & \underline{\rm aut}(A) \ar[d]^-{\mathcal H(\psi)}\\
\underline{\rm end}(A)\ar[r] & \underline{\rm aut}(A).
}
\]
In particular, for any graded algebra automorphism $\phi$ of $A$ and $\psi=\underline{\rm end}(\phi)$, we may write 
\[
\underline{\rm aut}(\phi):=\mathcal H(\underline{\rm end}(\phi)).
\]
Moreover, in view of \Cref{cor:lifting} and \Cref{Manin-twisting-pairs}, the twisting pair $(\underline{\rm end}^r((\phi^{-1})^!), \underline{\rm end}^l(\phi))$ for $\underline{\rm end}^l(A)$ can be lifted uniquely to a twisting pair for $\underline{\rm aut}^l(A)$, which we denote by $(\underline{\rm aut}^r((\phi^{-1})^!), \underline{\rm aut}^l(\phi))$. Similarly, we denote by $(\underline{\rm aut}^r(\phi), \underline{\rm aut}^l((\phi^{-1})^!))$ the unique lifting of the twisting pair $(\underline{\rm end}^r(\phi), \underline{\rm end}^l((\phi^{-1})^!))$ from $\underline{\rm end}^r(A)$ to $\underline{\rm aut}^r(A)$. 

\begin{thm}
\label{thm:twist-aut}
Let $A$ be a quadratic algebra and $\phi$ be any graded automorphism of $A$. Then we have Hopf algebra isomorphisms
\[
\underline{\rm aut}^l (A^\phi) \cong \underline{\rm aut}^l(A)^{\underline{\rm aut}^l(\phi)\circ \underline{\rm aut}^r((\phi^{-1})^!)}\qquad \text{and} \qquad \underline{\rm aut}^r (A^\phi) \cong \underline{\rm aut}^r(A)^{\underline{\rm aut}^r(\phi)\circ \underline{\rm aut}^l((\phi^{-1})^!)}.
\]
\end{thm}

\begin{proof}
This follows from \Cref{f-iso}, \Cref{prop:30}, and \Cref{thm:twist-end} since
{
\begin{align*}
\underline{\rm aut}^l(A^\phi)&\cong \mathcal H\left(\underline{\rm end}^l(A^\phi)\right)\cong \mathcal H(\underline{\rm end}^l(A)^{\underline{\rm end}^l(\phi)\circ \underline{\rm end}^r((\phi^{-1})^!)}) \\
& \cong \underline{\rm aut}^l(A)^{\mathcal H(\underline{\rm end}^l(\phi)\circ \uqba^r((\phi^{-1})^!))}
\cong \underline{\rm aut}^l(A)^{\underline{\rm aut}^l(\phi)\circ \underline{\rm aut}^r((\phi^{-1})^!)}.
\end{align*}
}
The argument for $\underline{\rm aut}^r(A^\phi)$ is similar. 
\end{proof}

\begin{thm}
\label{2cocycle-twist}
Let $A$ be a quadratic algebra and $\phi$ be any graded automorphism of $A$. Then 
\begin{enumerate}
    \item $\underline{\rm aut}^l(A^\phi)$ is a 2-cocycle twist of $\underline{\rm aut}^l(A)$ with the 2-cocycle $\sigma: \underline{\rm aut}^l(A)\otimes \underline{\rm aut}^l(A)\to \kk$ given by $\sigma(g,h)=\varepsilon(g)\varepsilon(\underline{\rm aut}^l(\phi)^{|g|}(h))$,  
    for any homogenous elements $g,h\in \underline{\rm aut}^l(A)$.
     \item $\underline{\rm aut}^r(A^\phi)$ is a 2-cocycle twist of $\underline{\rm aut}^r(A)$ with the 2-cocycle $\sigma: \underline{\rm aut}^r(A)\otimes \underline{\rm aut}^r(A)\to \kk$ given by $\sigma(g,h)= \varepsilon(g)\varepsilon(\underline{\rm aut}^l((\phi^{-1})^!)^{|g|}(h))$,  
    for any homogenous elements $g,h\in \underline{\rm aut}^r(A)$.
\end{enumerate} 
As a consequence, $\underline{\rm aut}(A)$ and $\underline{\rm aut}(A^\phi)$ are  Morita-Takeuchi equivalent.  
\end{thm}
\begin{proof}
This follows from \Cref{thm:Zhang and 2-cocycle twists}, \Cref{Manin-twisting-conditions}, \Cref{Manin-twisting-pairs}, and \Cref{thm:twist-aut}.
\end{proof}

%%%%%%%%%%%%%%%%%%%%%%%%%%%%%%%%%%%%%%%
\section{Twisting of solutions to the quantum Yang-Baxter equation}
\label{section-twist-qybe}

In this section, we apply 2-cocycle twists formed by twisting pairs to obtain a family of new solutions to the quantum Yang-Baxter equation (QYBE). Let $V$ be a finite-dimensional vector space over $\kk$ and $R\in {\rm End}_\kk(V^{\otimes 2})$ satisfy the QYBE. We first apply the Faddeev-Reshetikhin-Takhtajan (FRT) construction to obtain a quadratic bialgebra $A(R)$ that is equipped with a coquasitriangular structure given by $R$. By explicitly classifying all twisting pairs of $A(R)$, we find the corresponding 2-cocycles $\sigma$ on the Hopf envelope of $A(R)$ when it is the localization at some normal grouplike element $g$ (for instance, see \cite[Theorem 4.10]{MF}). Since the $2$-cocycle twist $(A(R)[g^{-1}])^{\sigma}$ preserves the braided structure of the comodule category over $A(R)[g^{-1}]$, it is again a coquasitriangular Hopf algebra with new twisted solutions $R^{\sigma}$ to the same QYBE. Readers may refer to \cite{CMZ, EGNO} for additional background on the QYBE and tensor categories.

%%%%%%%%%%%%%%%%%%%%%%%%%%%%%%%%%%%%%%%
\subsection{FRT constructions and their twisting pairs}
\label{subsect:twist-multilin}

We first review the FRT construction and the coquasitriangular structure on a bialgebra as given in \cite[Section 5.4]{CMZ}.

\begin{defn}\cite{CMZ}
Let $V$ be a finite-dimensional $\kk$-vector space and $R\in {\rm End}_{\kk}(V^{\otimes 2})$. Then $R$ is called a \emph{solution of the quantum Yang-Baxter equation} (QYBE) if 
\begin{equation}\label{eq:QYBE}
   R^{12}R^{13}R^{23}=R^{23}R^{13}R^{12},
\end{equation}
where $R=\sum R_1\otimes R_2\in {\rm End}_{\kk}(V^{\otimes 2}) \cong {\rm End}_{\kk}(V)^{\otimes 2}$ and $R^{12}=\sum R_1\otimes R_2\otimes \id$, $R^{13}=\sum R_1\otimes \id \otimes R_2$ and $R^{23}=\sum \id\otimes R_1\otimes R_2\in {\rm End}_{\kk}(V^{\otimes 3})$.
\end{defn}

Let $V$ be an $n$-dimensional $\kk$-vector space with fixed basis $\{x_1, \dots, x_n\}$, and let $R\in {\rm End}_\kk(V^{\otimes 2})$ be a solution to the QYBE. We may describe $R$ as an $(n^2\times n^2)$ matrix over $\kk$ with entries $R_{kl}^{ij} \in \kk$ such that
\begin{equation}
 R(x_k\otimes x_l)=\sum_{1\leq i,j\leq n} R_{kl}^{ij} \, x_i\otimes x_j, \quad \text{for } 1\leq k,l\leq n.
\label{coeff}   
\end{equation}
Then the FRT construction yields a bialgebra, denoted by $A(R)$, with explicit generators and relations. 
As an algebra, $A(R)$ is generated by $n^2$ generators $\{t^{j}_i\}_{i,j=1}^n$ subject to the following $n^4$-relations:
\begin{equation}\label{eq:AR}
    \sum_{1\leq k,l\leq n} R_{ij}^{kl}t^j_ut^i_v=\sum_{1\leq k,l\leq n}R^{ij}_{vu}t^k_it^l_j,
\end{equation}
for all $1\leq u,v,i,j\leq n$. The coalgebra structure on $A(R)$ is given by
\begin{equation}\label{eq:ARBi}
  \Delta(t_i^j)=\sum_{1\leq k\leq n} t^j_k\otimes t^k_i, \qquad \varepsilon(t_i^j)=\delta_i^j,
\end{equation}
where $\delta_i^j$ is the Kronecker delta.
 
 \begin{defn}\label{def:coquasitriangularbialgebra}\cite{CMZ}
A coquasitriangular (or braided) bialgebra (resp. Hopf algebra) is a pair ($B, \theta$), where $B$ is a bialgebra (resp. Hopf algebra) and $\theta: B\otimes B\longrightarrow \kk$ is a $\kk$-linear map such that
\begin{equation}\label{def:coquasitriangular}
    \left\{
\begin{aligned}
\sum \theta(x_1, y_1)y_2x_2&=\sum \theta(x_2, y_2)x_1y_1\\
\theta(x, 1)&=\varepsilon(x)\\
\theta(x, yz)&=\sum \theta(x_1, y)\theta(x_2, z)\\
\theta(1, x)&=\varepsilon(x)\\
\theta(xy, z)&=\sum \theta(y, z_1)\theta(x, z_2)
    \end{aligned}
    \right.
\end{equation}
for all $x,y,z\in B$. 
\end{defn}

Let $B$ be any bialgebra, and let ${\rm comod}(B)$ be the category of all right $B$-comodules. It is straightforward to check that $B$ has a coquasitriangular structure $\theta: B\otimes B\to \kk$ if and only if the tensor category ${\rm comod}(B)$ is braided. In such a case, the braiding 
\begin{equation*}
   \left \{
\tau_\theta(X,Y): X\otimes Y\to Y\otimes X \mid X, Y\in \textnormal{comod}(H) \right\}
\end{equation*}
is given by 
\begin{equation}\label{eq:braiding}
   \tau_{\theta}(X, Y)(x\otimes y)= \sum y_1\otimes x_1\theta(y_2,x_2)\quad \textnormal{for}\, x\in X, y\in Y.
\end{equation}

In the FRT construction $A(R)$ for some solution $R\in {\rm End}_\kk(V^{\otimes 2})$ of QYBE, we observe that $V$ becomes a right $A(R)$-comodule via $\rho: V\to V\otimes A(R)$ such that  
\begin{equation}
 \rho(x_i)=\sum_{j=1}^n x_j\otimes t^j_i   
\end{equation}
 for $1 \leq i\leq n$. The following results are well-known for the FRT construction. 
 
 \begin{proposition}\label{prop:AR}\cite{CMZ}
 Let $A(R)$ be defined as above. Then we have 
 \begin{enumerate}
     \item $R\circ \tau: V^{\otimes 2}\to V^{\otimes 2}$ is an $A(R)$-comodule map, where $\tau: V^{\otimes 2}\to V^{\otimes 2}$ is the usual flip map.
     \item Let $B$ be a bialgebra. If $V$ is a right $B$-comodule via $\rho': V\to V\otimes B$ and $R\circ \tau$ is also a $B$-comodule map, then exists a unique bialgebra map $f: A(R)\to B$ such that $(\id \otimes f)\circ \rho=\rho'$.
     \item $A(R)$ has a coquasitriangular structure $\theta: A(R)\otimes A(R)\to \kk$ where $\theta(t^i_v,t^j_u)=R^{ij}_{vu}$ for all $1\leq i,j,v,u\leq n$.
 \end{enumerate}
 \end{proposition}
 
We follow the notation used in the construction of $A(R)$ to write the entry in the $i$th row and the $j$th column of an $n \times n$ matrix $\mathbb A$ as $a^i_j$. The next result classifies all twisting pairs of $A(R)$. 

\begin{lemma} 
\label{Af-twisting-pair}
Let $A(R)$ be defined as above. Then every twisting pair $(\phi_1, \phi_2)$ of $A(R)$ is determined by its values on the generators $\{t_i^j\}_{1\leq i,j\leq n}$ such that  
\[ \phi_1(t_i^j)=\sum_{1\leq u \leq n}\alpha_i^ut_u^j \qquad {\rm and} \qquad\phi_2(t_i^j)=\sum_{1\leq u \leq n}\beta_{u}^jt^u_i,\] 
where the matrices $(\alpha_i^j)$
and $(\beta^v_u)$ are inverses of each other and  the following equations 
\[  \sum_{1\leq k,l\leq n} R_{ij}^{kl}\alpha^j_u\alpha^i_v=\sum_{1\leq k,l\leq n}R^{ij}_{vu}\alpha^k_i\alpha^l_j
\]
hold for all $1\leq u,v,i,j\leq n$. 
\end{lemma}
\begin{proof}
It is a direct computation by using \Cref{thm:twistingpair}, since any algebra homomorphism $A(R)\to \kk$ is determined by a map defined on the generators $(t^j_i)\mapsto (\alpha_i^j)$ that vanishes on the relations \eqref{eq:AR}. 
\end{proof}

We point out that by letting ${\rm deg}(t_i^j)=1$ for all $1 \leq i,j \leq n$ in the definition of $A(R)$, the bialgebra $A(R)$ becomes a quadratic algebra and satisfies the twisting conditions (\textbf{T1})-(\textbf{T2}) in \Cref{defn:twisting conditions}.

Now, let $n\ge 2$ be any integer and let $V$ be an $n$-dimensional vector space with a fixed basis $\{v_1,\ldots,v_n\}$. For any scalar $q \in \kk^{\times}$, we consider the classical Yang-Baxter operator $R_q$ on $V$ (cf. \cite{Takeuchi-qmatrix}) given by  
\begin{equation}\label{eq:YBO}
    R_q(v_i\otimes v_j)=\left \{\begin{aligned}
    &qv_i\otimes v_i &&\qquad i=j\\
    & v_i\otimes v_j &&\qquad i<j\\
    &v_i\otimes v_j+(q-q^{-1})v_j \otimes v_i &&\qquad i>j.
    \end{aligned}\right.
\end{equation}
In the following, we will classify all twisting pairs of the FRT construction $A(R_q)$.

\begin{Ex}\label{ex:gl}
Let $n\ge 2$ be an integer and let $q\in \kk^\times $ be any scalar. The one-parameter quantization of the coordinate ring of $n\times n$ matrices $\mathcal O_q(M_n(\kk))$ is the algebra generated by $n^2$-generators $\{x_{ij}\}_{1\leq i,j\leq n}$ subject to the relations (see \cite[Section 1]{Zhang98}):
 \begin{equation}
     \label{eq:Rqrelation1}
     \left\{\begin{aligned}
qx_{ks} x_{us} &= x_{us} x_{ks}&&{\rm if }\,\, k<u\\ 
qx_{ks} x_{kv} &= x_{kv} x_{ks} &&{\rm if }\,\,  s<v\\
x_{us} x_{kv} &= x_{kv} x_{us}&&{\rm if }\,\,  s<v, k<u\\
x_{us} x_{kv} &= x_{kv} x_{us}+(q-q^{-1})x_{ks} x_{uv} &&{\rm if }\,\, s<v,
u<k.
 \end{aligned}\right.
 \end{equation}
 Moreover, $\mathcal O_q(M_n(\kk))$ is a bialgebra with coalgebra structure given by
 \begin{equation}
     \Delta(x_{ij})=\sum_{1\leq k\leq n} x_{ik}\otimes x_{kj}\quad \text{and}\quad \varepsilon(x_{ij})=\delta_{ij}\quad \text{for all}\ 1\leq i,j\leq n. 
 \end{equation}
In particular, there is a central group-like element $g$ of $\mathcal O_q(M_n(\kk))$, called the $q$-determinant, defined by 
\begin{equation} 
\label{central g}
  g:=\sum_{\sigma\in S_n}(-q)^{-l(\sigma)}x_{\sigma(1)1} \cdots x_{\sigma(n)n}, 
\end{equation} 
where $l(\sigma)$ denotes the length of the permutation $\sigma$, that is, the minimum length of an expression for $\sigma$ as a product of adjacent transpositions $(i, i+1)$. 
\end{Ex}

\begin{proposition}\cite[Proposition 4.3, Definition 4.5, Lemma 4.8, Corollary 4.10]{Takeuchi-qmatrix}
Retain the notations above. Then we have 
\begin{enumerate}
    \item $A(R_q)$ is isomorphic to $\mathcal O_q(M_n(\kk))$ as coquasitriangular bialgebras. 
    \item The central group-like element of $A(R_q)$ corresponds to the $q$-determinant of $\mathcal O_q(M_n(\kk))$.
    \item  The Hopf envelope of $A(R_q)$ is isomorphic to $\mathcal O_q(\GL_n(\kk))$. 
\end{enumerate} 
\end{proposition}

 The Hopf envelope of $\mathcal O_q(M_n(\kk))$ is the localization of $\mathcal O_q(M_n(\kk))$ at $g$, which is usually denoted as $\mathcal O_q(\GL_n(\kk))$, called the one-parameter quantization of the coordinate ring of the general linear group. 

Now, we classify all twisting pairs of $\mathcal O_q(\GL_n(\kk))$ and give explicit formulas of the corresponding 2-cocycles. 

\begin{proposition}
\label{quantizedGL-twisting-pair}
All twisting pairs of $\mathcal O_q(\GL_n(\kk))$ are of the form $(\phi_1, \phi_2)$, where $\phi_1$ and $\phi_2$ are given by
 \[
 \begin{pmatrix*}[c]
\phi_1(x_{11})         & \cdots  & \phi_1(x_{1n}) \\
\vdots               & \ddots   & \vdots             \\
\phi_1(x_{n1})         & \cdots  & \phi_1(x_{nn}) 
\end{pmatrix*}=
 \begin{pmatrix*}[c]
x_{11}         & \cdots  & x_{1n} \\
\vdots               & \ddots   & \vdots             \\
x_{n1}         & \cdots  & x_{nn} 
\end{pmatrix*}
\begin{pmatrix*}[c]
\alpha_{11}         & \cdots  & \alpha_{1n} \\
\vdots               & \ddots   & \vdots             \\
\alpha_{n1}         & \cdots  & \alpha_{nn} 
\end{pmatrix*}
 \]
 and 
 \[
 \begin{pmatrix*}[c]
\phi_2(x_{11})         & \cdots  & \phi_2(x_{1n}) \\
\vdots               & \ddots   & \vdots             \\
\phi_2(x_{n1})         & \cdots  & \phi_2(x_{nn}) 
\end{pmatrix*}= \begin{pmatrix*}[c]
\alpha_{11}         & \cdots  & \alpha_{1n} \\
\vdots               & \ddots   & \vdots             \\
\alpha_{n1}         & \cdots  & \alpha_{nn} 
\end{pmatrix*}^{-1}
 \begin{pmatrix*}[c]
x_{11}         & \cdots  & x_{1n} \\
\vdots               & \ddots   & \vdots             \\
x_{n1}         & \cdots  & x_{nn} 
\end{pmatrix*},
 \]
with $(\alpha_{ij})\in \GL_n(\kk)$. Moreover, we have the following.
\begin{enumerate}
\item If $q=1$, then for any $(\alpha_{ij}) \in \GL_n(\kk)$, the maps $(\phi_1, \phi_2)$ defined above form a twisting pair.
\item If ${\rm char}(\kk)\neq 2$ and $q=-1$, then $(\alpha_{ij})$ defines a twisting pair as above if and only if it is a generalized permutation matrix. 
\item If $q \neq \pm 1$, then $(\alpha_{ij})$ defines a twisting pair as above if and only if it is is diagonal.
\end{enumerate}
 In particular, $\phi_1(g^{-1})=(-1)^{l(\tau)}|(\alpha_{ij})|^{-1}g^{-1}$, $\phi_2(g^{-1})=(-1)^{l(\tau)}|(\alpha_{ij})|g^{-1}$, and $\phi_1 \circ \phi_2(g^{-1})=g^{-1}$, where $\tau=\id$ if $q\ne -1$, and $\tau\in S_n$ such that $|(\alpha_{ij})|=(-1)^{l(\tau)}\alpha_{\tau(1)1}\cdots \alpha_{\tau(n)n}$ if $q=-1$. 
\end{proposition}

\begin{proof}
By \Cref{Af-twisting-pair}, we can write $\phi_1$ and $\phi_2$ as above, represented by two matrices $(\alpha_{ij})$ and $(\beta_{ij})$ that are inverse to each other. In particular, the map $\pi: (x_{ij})\mapsto (\alpha_{ij})$ is an algebra map from $\mathcal O_q(\GL_n(\kk))\to \kk$. 

In the following, we use the explicit relations in $\mathcal O_q(\GL_n(\kk))$ stated in \Cref{ex:gl} to determine $\pi$.
 
When $q=1$, there are no restrictions on $(\alpha_{ij})$. 

When $q=-1$ and ${\rm char}(\kk)\neq 2$, we have $\alpha_{ks}\alpha_{us}=\alpha_{sk}\alpha_{su}=0$ whenever $k\neq u$. This implies that $(\alpha_{ij})$ is a generalized permutation matrix. 

When $q\neq \pm 1$, we have $\alpha_{ks}\alpha_{us}=\alpha_{sk}\alpha_{su}=0$ whenever $k\neq u$ and $\alpha_{ks}\alpha_{uv}=0$ whenever $s<v$ and $u<k$. The first part implies that $(\alpha_{ij})$ is a generalized permutation matrix as before. Now suppose we have some $\alpha_{uv}\neq 0$ for some $u\neq v$. Without loss of generality, say $u<v$. So we get $\alpha_{ks}=0$ for all $s<v$ and $u<k$. This implies that the nonzero entries in the last $n-u$ rows should be concentrated in the last $n-v$ columns regarding the matrix $(\alpha_{ij})$. This is a contradiction since it is a generalized permutation matrix and $n-u>n-v$.  Therefore, $\alpha_{uv}\ne 0$ only if $u=v$ and so $(\alpha_{ij})$ is indeed diagonal. 

Next, we describe the maps $\phi_1, \phi_2$ on $g$. For $X=(x_{ij})$, we can write $\phi_1(X)=(\phi_1(x_{ij}))=X\cdot (\alpha_{ij})$ and $\phi_2(X)=(\alpha_{ij})^{-1}\cdot X$. By \eqref{central g}, $g=|X|_q$, so $\phi_1(g)=\phi_1(|X|_q)=|\phi_1(X)|_q$.
One can see that
\begin{align*}
   |X \cdot (\alpha_{ij})|_q&=\left|\left(\sum_{1\leq k\leq n}x_{ik}\alpha_{kj}\right)\right|_q\\
   &=\sum_{\sigma\in S_n}(-q)^{-l(\sigma)}\left(\sum_{1\leq k_1\leq n }x_{\sigma(1)k_1}\alpha_{k_11}\right)\cdots \left(\sum_{1\leq k_n\leq n }x_{\sigma(n)k_n}\alpha_{k_nn}\right) \\
   &=\sum_{\sigma\in S_n}(-q)^{-l(\sigma)}\sum_{1\leq k_\bullet \leq n }x_{\sigma(1)k_1}\cdots x_{\sigma(n)k_n}\alpha_{k_11}\cdots \alpha_{k_nn}.
\end{align*}
If $q=1$, then $|X\cdot (\alpha_{ij})|=|X||(\alpha_{ij})|=|(\alpha_{ij})|g$.
If $q\ne \pm 1$, then since $(\alpha_{ij})$ is diagonal, we have \[\phi_1(g)=\sum_{\sigma\in S_n}(-q)^{-l(\sigma)}x_{\sigma(1)1}\cdots x_{\sigma(n)n}\alpha_{11}\cdots \alpha_{nn}=|(\alpha_{ij})|g.\]
If $q=-1,$ then $(\alpha_{ij})$ is a generalized permutation matrix, 
and so $|(\alpha_{ij})|=(-1)^{l(\tau)}\alpha_{\tau(1)1}\cdots \alpha_{\tau(n)n}$ for some $\tau\in S_n$. 
Therefore, for this $\tau$ we have
\begin{align*}
    \phi_1(g)&=\sum_{\sigma\in S_n}\sum_{1\leq k_\bullet \leq n }x_{\sigma(1)k_1}\cdots x_{\sigma(n)k_n}\alpha_{k_1 1}\cdots \alpha_{k_n n}\\
    &=\sum_{\sigma\in S_n}x_{\sigma(1)\tau(1)}\cdots x_{\sigma(n)\tau(n)}\alpha_{\tau(1)1}\cdots \alpha_{\tau(n)n}\\
    &=(-1)^{l(\tau)}|(\alpha_{ij})|\sum_{\sigma\in S_n}x_{\sigma(1)\tau(1)}\cdots x_{\sigma(n)\tau(n)}.
    \end{align*}
By \eqref{eq:Rqrelation1}, $x_{us}x_{kv}=x_{kv}x_{us}$ if $u\ne k,$ and $s\ne v$. Thus, \[\phi_1(g)=(-1)^{l(\tau)}|(\alpha_{ij})|\sum_{\sigma\in S_n}x_{\sigma(1)1}\cdots x_{\sigma(n)n}=(-1)^{l(\tau)}|(\alpha_{ij})|g.\] Similarly, we have $\phi_2(g)=|(\alpha_{ij})^{-1}|g$ if $q\ne -1$  and $\phi_2(g)=(-1)^{l(\tau)}|(\alpha_{ij})^{-1}|g$ if $q=-1$. 
Since $\phi_1$ and $\phi_2$ are algebra maps, we conclude that the images of $\phi_1(g^{-1})$, $\phi_2(g^{-1})$ and $\phi_1\circ \phi_2(g^{-1})$ are as in the statement.
\end{proof}

\begin{Cor}\label{cor:cocycleGL}
Retain the notation in \Cref{quantizedGL-twisting-pair}. The right Zhang twist of $\mathcal O_q(\GL_n(\kk))$ by $\phi_1\circ \phi_2$ is isomorphic to the 2-cocycle twist of $\mathcal O_q(\GL_n(\kk))$ by the 2-cocycle $\sigma$ such that 
\begin{align*}  
&\sigma(x_{i_1j_1}^{p_1}\cdots x_{i_sj_s}^{p_s}g^{-r}, x_{u_1v_1}^{q_1}\cdots x_{u_kv_k}^{q_k}g^{-t})\\
&= (-1)^{mtl(\tau)}\delta_{i_1, j_1}\cdots \delta_{i_s, j_s}\left(((\alpha_{ij})^{-m})_{u_1v_1}\right)^{q_1} \cdots \left(((\alpha_{ij})^{-m})_{u_kv_k}\right)^{q_k}|(\alpha_{ij})|^{mt},
\end{align*}
where $m=p_1+\dots +p_s-nr$ and $1\leq i_\bullet,j_\bullet,u_\bullet,v_\bullet\leq n$ and $p_\bullet, q_\bullet,s,k, r,t$ are non-negative integers; and  $\tau=\id$ if $q\ne -1$ and $\tau\in S_n$ such that $|(\alpha_{ij})|=(-1)^{l(\tau)}\alpha_{\tau(1)1}\cdots \alpha_{\tau(n)n}$ if $q=-1$.
\end{Cor}
\begin{proof}
We know $\phi_2((x_{ij}))=(\alpha_{ij})^{-1}(x_{ij})$. So it follows from \Cref{thm:Zhang and 2-cocycle twists} and \Cref{quantizedGL-twisting-pair} that 
\begin{align*} 
&\sigma(x_{i_1j_1}^{p_1}\cdots x_{i_sj_s}^{p_s}g^{-r}, x_{u_1v_1}^{q_1}\cdots x_{u_kv_k}^{q_k}g^{-t})\\
&= \varepsilon(x_{i_1j_1}^{p_1}\cdots x_{i_sj_s}^{p_s}g^{-r})\varepsilon(\phi_2^m(x_{u_1v_1}^{q_1}\cdots x_{u_kv_k}^{q_k}g^{-t})) \\
&= \delta_{i_1j_1}\cdots \delta_{i_sj_s}\varepsilon(\phi_2^m(x_{u_1v_1}))^{q_1}\cdots \varepsilon(\phi_2^m(x_{u_kv_k}))^{q_k}\varepsilon(\phi_2^m(g^{-1}))^{t}\\
&= \delta_{i_1j_1}\cdots \delta_{i_sj_s}\varepsilon\left(\sum_{1\leq w \leq n}\left((\alpha_{ij})^{-m}\right)_{u_1w}x_{wv_1}\right)^{q_1} \hspace{-0.1in} \cdots \varepsilon\left(\sum_{1\leq w \leq n}\left((\alpha_{ij})^{-m}\right)_{u_kw}x_{wv_k}\right)^{q_k} \hspace{-0.1in} \varepsilon\left({((-1)^{l(\tau)}|(\alpha_{ij})|)^{m}}g^{-1}\right)^{t}\\
&= (-1)^{mtl(\tau)}\delta_{i_1, j_1}\cdots \delta_{i_s, j_s}\left(((\alpha_{ij})^{-m})_{u_1v_1}\right)^{q_1}\cdots \left(((\alpha_{ij})^{-m})_{u_kv_k}\right)^{q_k}|(\alpha_{ij})|^{mt},
\end{align*} 
where $\tau=\id$ if $q\ne -1$, and $\tau\in S_n$ such that $|(\alpha_{ij})|=(-1)^{l(\tau)}\alpha_{\tau(1)1}\cdots \alpha_{\tau(n)n}$ if $q=-1$.
\end{proof}

%%%%%%%%%%%%%%%%%%%%%%%%%%%%%%%%%%%%%%%
\subsection{Twisting of solutions to the quantum Yang-Baxter equation}

In this subsection, we propose an algorithm to find new solutions to the quantum Yang-Baxter equation by applying twisting pairs to the FRT construction.

Let $H$ be any Hopf algebra with some 2-cocycle $\sigma: H\otimes H\longrightarrow \kk$. Then we have the following Morita-Takeuchi equivalence:
\[
{\rm comod}(H)\overset{\otimes}{\cong} {\rm comod}(H^\sigma),
\]
where the tensor equivalence functor $(F,F_1,F_2): {\rm comod}(H)\overset{\otimes}\to {\rm comod}(H^\sigma)$ is given by $F=\id$ is the identity functor, $F_1=\id: F(\kk)=\kk$, and 
\[
\{F_2(X\otimes Y): F(X)\otimes_\sigma F(Y)\to F(X\otimes Y)\,|\, X,Y\in {\rm comod}(H)\}\]
is given by
\begin{equation}
x\otimes_{\sigma}y\mapsto \sum x_1\otimes y_1\sigma^{-1}(x_2,y_2), \quad \text{for any}\,\, x\in X,y\in Y.
\end{equation}

The following result is well-known (e.g., see the discussion in \cite[p. 61]{MS1995}) and we include the proof for the convenience of the reader.

\begin{lemma}\label{thm1s3}
Let $(H, \theta)$ be a coquasitriangular Hopf algebra and $\sigma: H\otimes H\longrightarrow 
\kk$ be a 2-cocycle on the $H$.
Then the 2-cocycle twist $H^\sigma$ is again a coquasitriangular Hopf algebra with the new  coquasitriangular structure $\theta^\sigma: H^\sigma\otimes H^\sigma\to \kk$ given by
\[\theta^\sigma(g, h)=\sigma(g_1, h_1)\theta(g_2, h_2)\sigma^{-1}(h_3, g_3), \qquad \text{for all } g,h\in H^\sigma.\] 
\end{lemma}
\begin{proof}
As discussed above, having a coquasitriangular structure on a Hopf algebra is equivalent to having a braided structure on its comodule category. Since the comodule categories of $H$ and of its 2-cocycle twist $H^\sigma$ are tensor equivalent \cite{DT94}, we can transfer the braided structure from ${\rm comod}(H)$ to ${\rm comod}(H^\sigma)$ by applying \eqref{eq:braiding}. We define the braiding $\tau^\sigma$ on any two objects $F(X), F(Y)\in {\rm comod}(H^\sigma)$ to make the following diagram 
\begin{equation*}\label{def:braiding}
\xymatrix{
F(V)\otimes_{\sigma} F(W)\ar@{-->}[r]^-{\tau_{V, W}^{\sigma}}\ar[d]_-{F_2(V, W)} & F(W)\otimes_{\sigma} F(V) \\
F(V\otimes W)\ar[r]^-{F(\tau_{\theta}(V, W))}&  F(W\otimes V) \ar[u]_-{F_2(W, V)^{-1}}
}
\end{equation*}
commute. Then it is straightforward to check that \begin{align*}
    \tau^{\sigma}_{V, W}(v\otimes_{\sigma}w)=\sum w_1\otimes v_1\sigma(w_2, v_2)\theta(w_3, v_3)\sigma^{-1}(v_4, w_4)=\sum w_1\otimes v_1\theta^\sigma(w_2,v_2)
\end{align*}
for any $v\in V,w\in W$. This proves our result. 
 \end{proof}

\begin{proposition}
Let $(H,\theta)$ be a coquasitriangular Hopf algebra. Suppose $H$ satisfies the twisting conditions (\textbf{T1})-(\textbf{T2}). Then for any twisting pair $(\phi_1,\phi_2)$ of $H$, we have  
\[
\theta^\sigma(x,y)=\theta(\phi_1^{|y|}(x),\phi_2^{|x|}(y)),\quad \text{for all homogeneous elements}\ x,y\in H^{\sigma} 
\]
is a coquasitriangular structure on the 2-cocycle twist $H^{\sigma}$ where $\sigma: H\otimes H\to \kk$ is given by $\sigma(x,y)=\varepsilon(x)\varepsilon(\phi_2^{|x|}(y))$ for all homogeneous elements $x,y\in H$.
\end{proposition}
\begin{proof}
By our main result \Cref{twist-and-2-cocycle},  $\sigma$ is a well-defined 2-cocycle on $H$ with convolution inverse $\sigma^{-1}(x, y)=\varepsilon(x)\varepsilon(\phi_1^{|x|}(y))$ for all homogeneous elements $x,y\in H$. Thus, by \Cref{thm1s3}, we get 
\begin{align*}
  \theta^\sigma(v, w)&=\sum \sigma(v_1, w_1)\theta(v_2, w_2)\sigma^{-1}(w_3, v_3)\\
  &=\sum \varepsilon(v_1)\varepsilon(\phi_2^{|v_1|}(w_1))\theta(v_2, w_2)\varepsilon(w_3)\varepsilon(\phi_1^{|w_3|}(v_3))\\
  &=\sum \varepsilon(\phi_2^{|v|}(w_1))\theta(v_1, w_2)\varepsilon(\phi_1^{|w|}(v_2))\\
  &=\theta\left(\sum v_1\varepsilon(\phi_1^{|w|}(v_2)),\, \sum  \varepsilon(\phi_2^{|v|}(w_1))w_2\right)\\
    \overset{(\textbf{P1})}&=\theta\left(\sum \phi_1^{|w|}(v)_1\varepsilon(\phi_1^{|w|}(v)_2),\, \sum  \varepsilon(\phi_2^{|v|}(w)_1)\phi_2^{|v|}(w)_2\right)\\
  &=\theta(\phi_1^{|w|}(v), \phi_2^{|v|}(w)). \qedhere
\end{align*}
\end{proof}

Let $V$ be an $n$-dimensional vector space for some positive integer $n\ge 2$, and let $R\in {\rm End}_\kk(V^{\otimes 2})$ be a solution to the quantum Yang-Baxter equation. Recall the FRT construction $A(R)$ by \eqref{eq:AR} and \eqref{eq:ARBi}. 

\begin{defn}
Retain the notations above. For any 2-cocycle $\sigma$ on $A(R)$, we define the twisted $R^\sigma\in {\rm End}_\kk(V^{\otimes 2})$ as 
\[
(R^\sigma)^{kl}_{ij}:=\sum_{1\leq u,v,p,q\leq q} \sigma(t^k_u,t^l_v)R^{uv}_{pq}\sigma^{-1}(t^q_j,t^p_i), \qquad \text{for all } 1\leq i,j,k,l \leq n.\] 
\end{defn}

The following result is straightforward.

\begin{lemma}
Let $\sigma$ be a 2-cocycle twist on $A(R)$. Then we have 
\[
A(R)^\sigma\cong A(R^\sigma)
\]
as bialgebras. As a consequence, $R^\sigma\in {\rm End}_\kk(V^{\otimes 2})$ is another solution to the quantum Yang-Baxter equation.
\end{lemma}

\begin{Cor}
\label{cor:TwistR}
Retain the notations above. Suppose the 2-cocycle $\sigma$ on $A(R)$ is given by some twisting pair $(\phi_1, \phi_2)$ as in  \Cref{Af-twisting-pair}. Then we have  
\[\left(R^{\sigma}\right)^{kl}_{ij}=\sum_{1\leq p, v\leq n}\alpha^p_iR^{kv}_{pj}\beta^l_v, \quad \text{ for all } 1\leq i,j,k,l\leq n, \]
is another solution to the quantum Yang-Baxter equation in ${\rm End}_\kk(V^{\otimes 2})$.
\end{Cor}
\begin{proof}
It is a straightforward computation:
\begin{align*}
(R^\sigma)^{kl}_{ij}&=\sum_{1\leq u,v,p,q\leq n} \sigma(t^k_u,t^l_v)R^{uv}_{pq}\sigma^{-1}(t^q_j,t^p_i) \\
&= \sum_{1\leq u,v,p,q\leq n}\varepsilon(t^k_u)\varepsilon(\phi_2(t^l_v))R^{uv}_{pq} \varepsilon(t^q_j)\varepsilon(\phi_1(t^p_i))\\
&=\sum_{1\leq v,p\leq n} \varepsilon(\phi_1(t^p_i))R^{kv}_{pj}\varepsilon(\phi_2(t^l_v)) =\sum_{1\leq p, v\leq n}\alpha^p_iR^{kv}_{pj}\beta^l_v. \qedhere
\end{align*} 
\end{proof}

Finally we apply our previous results to find all twisting solutions to the quantum Yang-Baxter equation from the classical Yang-Baxter operator $R_q$.   

\begin{proposition}
\label{prop:twistR}
Let $R_q\in \textnormal{End}_{\kk}(V^{\otimes 2})$ be the classical Yang-Baxter operator as in \eqref{eq:YBO}. Let $(\phi_1,\phi_2)$ be a twisting pair of $A(R_q)$ as given in  \Cref{Af-twisting-pair}. Denote by $\sigma$  the 2-cocycle corresponding to the right Zhang twist of $\phi_1\circ \phi_2$. Then the new solution $(R_q)^{\sigma}$ to the QYBE in \Cref{cor:TwistR} is described as follows:
\begin{enumerate}
    \item If $q=1$, then $\left((R_q)^{\sigma}\right)^{kl}_{ij}=\alpha^k_i\beta^l_j$.
    \item If $q=-1$, then $ \left((R_q)^{\sigma}\right)^{kl}_{ij}=(-1)^{\delta_j^k}\,\delta^k_{\tau^{-1}(i)}\delta^l_{\tau(j)}\alpha^k_i(\alpha_l^j)^{-1}$, for some permutation $\tau\in S_n$ such that $\alpha^{i}_j\neq 0$ whenever $j=\tau(i)$.
    \item If $q\ne \pm1$, then $ \left((R_q)^{\sigma}\right)^{kl}_{ij}=\left(R_q\right)_{ij}^{kl}\alpha_i^i\,(\alpha_l^l)^{-1}$. 
\end{enumerate}
\end{proposition}

\begin{proof}
Write the $R_q$-matrix as follows
\begin{equation*}
\left(R_q\right)_{ij}^{kl}=\left\{
    \begin{aligned}
&\delta^{i}_{k}\delta^{j}_{l}, & \textnormal{if}\quad  i<j\\
&\delta^{i}_{k}\delta^{j}_{l}\,q,
& \textnormal{if}\quad  i=j\\
&\delta^{i}_{k}\delta^{j}_{l}+\delta^{i}_{l}\delta_{j}^{k}(q-q^{-1}), 
& \quad \textnormal{if}\quad i>j.
\end{aligned}
\right.
\end{equation*}
Recall that $\phi_1(t_i^j)=\sum_{1\leq u \leq n}\alpha_i^ut_u^j$ and $ \phi_2(t_i^j)=\sum_{1\leq u \leq n}\beta_{u}^jt^u_i$ where $(\alpha^u_i)$ and $(\beta^j_u)$ are $n\times n$ matrices inverses of each other. We apply Corollary \ref{cor:TwistR} in the following three cases.

{\bf Case 1}:
If $q=1$, then
we have
\[
\left((R_q)^{\sigma}\right)^{kl}_{ij}=\sum_{1\leq p, v\leq n}\alpha^p_iR^{kv}_{pj}\beta^l_v=\alpha^k_iR^{kj}_{kj}\beta^l_j=\alpha^k_i\beta^l_j.
\]

{\bf Case 2}: 
If $q=-1$, then we have
\[
\left((R_q)^{\sigma}\right)^{kl}_{ij}=\sum_{1\leq p, v\leq n}\alpha^p_iR^{kv}_{pj}\beta^l_v=\alpha^k_iR^{kj}_{kj}\beta^l_j=(-1)^{\delta_j^k}\,\alpha^k_i\beta^l_j. 
\]
Since $(\alpha^i_j)$ is a generalized permutation matrix, there is some $\tau\in S_n$ such that $\alpha^i_j\neq 0$ if $j=\tau(i)$ and $\alpha^i_j=0$ if $j\neq \tau(i)$. Then $\beta^i_j=(a^j_i)^{-1}$ if $i=\tau(j)$ and $\beta^i_j=0$ if $i\neq \tau(j)$. Hence, 
\[
\left((R_q)^{\sigma}\right)^{kl}_{ij}=(-1)^{\delta_j^k}\,\delta^k_{\tau^{-1}(i)}\delta^l_{\tau(j)}\alpha^k_i(\alpha_l^j)^{-1}.
\]

{\bf Case 3}:
If $q\ne \pm 1$, then  $(\alpha_{i}^{j})$ and $(\beta^j_i)$ are invertible diagonal matrices. Then we have  
\[
\left((R_q)^{\sigma}\right)^{kl}_{ij}=\sum_{1\leq p, v\leq n}\alpha^p_iR^{kv}_{pj}\beta^l_v=\alpha_i^iR^{kl}_{ij}\beta_l^l=\left\{\begin{aligned}
&\delta^{i}_{k}\delta^{j}_{l}\,\alpha_i^i\,(\alpha_l^l)^{-1}, & \textnormal{if}\quad  i<j\\
&\delta^{i}_{k}\delta^{j}_{l}\,\alpha_i^i\,(\alpha_l^l)^{-1}\,q,
& \textnormal{if}\quad  i=j\\
&\left(\delta^{i}_{k}\delta^{j}_{l}+\delta^{i}_{l}\delta_{j}^{k}(q-q^{-1})\right)\alpha_i^i\,(\alpha_l^l)^{-1}, 
& \quad \textnormal{if}\quad i>j. 
\end{aligned}
\right.  
\]
\end{proof}
%%%%%%%%%%%%%%%%%%%%%%%%%%%%%%%%%%%%%%%%%
\section*{Conflict of interest}

The authors declare that they have no conflict of interest.

%%%%%%%%%%%%%%%%%%%%%%%%%%%%%%%%%%%%%%%%%
\bibliography{bibl}
\bibliographystyle{plain}
\end{document}